\documentclass[12pt]{article} 
\usepackage{arxiv}
\usepackage{comment}
\usepackage{times}
\usepackage{thmtools}
\usepackage{amsmath}
\usepackage{graphicx}
\usepackage{amsthm}
\usepackage{natbib}
\usepackage{doi}
\usepackage{proof}
\usepackage{microtype} 
\usepackage{collect}
\title{Online Newton Method for Bandit Convex Optimisation}
\author{
Hidde Fokkema\\
Korteweg-de Vries Institute for Mathematics\\
University of Amsterdam\\
\texttt{h.j.fokkema@uva.nl}
\And
Dirk van der Hoeven\\
Mathematical Institute\\
Leiden University\\
\texttt{dirk@dirkvanderhoeven.com}
\And
Tor Lattimore\\
Google DeepMind\\
\texttt{lattimore@google.com}
\And
Jack J. Mayo\\
Korteweg-de Vries Institute for Mathematics\\
University of Amsterdam\\
\texttt{jackjamesmayo@gmail.com}
}
\date{}

\usepackage{xcolor}
\definecolor{dkblue}{cmyk}{1,.54,.04,.19}

\usepackage[textsize=tiny]{todonotes}

\usepackage{hyperref} 

\hypersetup{
  bookmarks=true,         
    unicode=false,          
    pdftoolbar=true,        
    pdfmenubar=true,        
    pdffitwindow=false,     
    pdfstartview={FitH},    
    pdftitle={A Second-Order Method for Constrained Stochastic Bandit Convex Optimisation},    
    pdfauthor={authors},
    pdfsubject={Bandits},   
    pdfcreator={pdflatex},   
    pdfproducer={Producer}, 
    pdfkeywords={bandits} {zeroth-order convex optimisation} {machine learning}, 
    pdfnewwindow=true,      
    colorlinks=true,        
    linkcolor=black,       
    citecolor=dkblue,       
    filecolor=dkblue,       
    urlcolor=dkblue,        
}

\usepackage{enumitem}
\setlist[enumerate,1]{%
    label={\normalfont\textit{(\alph*)}},
    itemsep=0pt,
    labelwidth=!,
    labelindent=10pt
}
\usepackage{times}
\usepackage{pgfplots}
\pgfplotsset{compat=1.11}
\usepackage{pgfplotstable}
\usepackage{multicol}
\usepackage{nicefrac}
\usepackage{colortbl}
\usepackage{floatrow}
\usepackage{amsmath}
\usepackage{setspace}
\usepackage{mathtools}
\usepackage{listings}
\usepackage{mathrsfs}
\usepackage{soul}
\usepackage{tikz}
\usepackage{wrapfig}
\usepackage{dsfont}
\usepackage{amssymb}
\usepackage[boxed]{algorithm}
\usepackage{bm}
\usepackage[capitalize]{cleveref}
\usepackage[bf]{caption}
\usepackage{graphicx}

\let\epsilon\varepsilon

\usepackage{thmtools, thm-restate} 
\newtheorem{theorem}{Theorem}
\newtheorem{lemma}[theorem]{Lemma}
\newtheorem{proposition}[theorem]{Proposition}
\newtheorem{corollary}[theorem]{Corollary}
\newtheorem{remark}[theorem]{Remark}
\theoremstyle{remark}

\theoremstyle{definition}
\newtheorem{definition}[theorem]{Definition}

\usepackage{regexpatch}
\makeatletter
\xpatchcmd\thmt@restatable{%
\csname #2\@xa\endcsname\ifx\@nx#1\@nx\else[{#1}]\fi
}{%
\ifthmt@thisistheone
\csname #2\@xa\endcsname\ifx\@nx#1\@nx\else[{#1}]\fi
\else
\csname #2\@xa\endcsname[{\textsc{restated}}]
\fi}{}{}
\makeatother

\usepackage{natbib}

\newcommand{\constantCheck}{} 

\newcommand{\R}{\mathbb R}
\newcommand{\argmin}{\operatornamewithlimits{arg\,min}}

\newcommand{\explan}[1]{\stackrel{\text{\tiny \texttt{#1}}}}
\newcommand{\ip}[1]{\left \langle #1 \right \rangle}

\newcommand{\sip}[1]{\langle #1 \rangle}
\newcommand{\Reg}{\textup{\textrm{Reg}}}

\newcommand{\Fmax}{\textup{\textrm{F}}_{\max}}

\newcommand{\sphere}{\mathbb{S}}
\newcommand{\logs}{L}
\newcommand{\ball}{\mathbb{B}}
\newcommand{\norm}[1]{\left \Vert  #1 \right \Vert}

\newcommand\half[1]{\textstyle{\frac{#1}{2}}}
\newcommand\pip{\pi_{\scalebox{0.6}{+}}}
\newcommand{\ext}{\operatorname{e}}

\newcommand{\E}{\mathbb E}

\newcommand{\cF}{\mathscr F}

\newcommand{\sF}{\mathscr F}

\newcommand{\cN}{\mathcal N}

\newcommand{\tr}{\operatorname{tr}}

\newcommand{\zeros}{ \bm 0}

\newcommand{\bbP}{\mathbb P}

\newcommand{\lip}{\operatorname{Lip}}

\newcommand{\polylog}{\operatorname{polylog}}

\newcommand{\diam}{\operatorname{diam}}
\newcommand{\poly}{\operatorname{poly}}
\newcommand{\sind}{\bm{1}}
\renewcommand{\d}[1]{\operatorname{d}\!#1}
\newcommand{\id}{\mathds{1}}

\newcommand{\Ymax}{Y_{\max}}

\newcommand{\xopt}{x_{\star}}
\newcommand{\sReg}{\Reg^{\small s}}
\newcommand{\fReg}{\Reg^{\small f}}
\newcommand{\qReg}{\Reg^{\small q}}
\newcommand{\hqReg}{\Reg^{\small \hat q}}

\newcommand{\eventGaussian}{\textsc{e1}}
\newcommand{\eventCov}{\textsc{e2}}
\newcommand{\eventConc}{\textsc{e3}}
\newcommand{\eventFSquared}{\textsc{e4}}
\newcommand{\eventGood}{\textsc{e0}}

\newcommand{\TODO}[1]{%
\ifmmode
\text{\textcolor{red}{TODO: #1}}
\else
\textcolor{red}{TODO: #1}
\fi
}

\setcitestyle{square}

\begin{document}

\maketitle

\begin{abstract}
We introduce a computationally efficient algorithm for zeroth-order bandit convex optimisation and
prove that in the adversarial setting its regret is at
most $d^{3.5} \sqrt{n} \polylog(n, d)$ with high probability where $d$ is the dimension and $n$ is the time horizon.
In the stochastic setting the bound improves to $M d^{2} \sqrt{n} \polylog(n, d)$ where $M \in [d^{-1/2}, d^{-1 / 4}]$ is a constant
that depends on the geometry of the constraint set and the desired computational properties.
\end{abstract}

\section{Introduction}
Bandit convex optimisation is the bandit version of the classical zeroth-order optimisation problem, which is both a fundamental problem in optimisation and has many
applications in operations research and beyond.
Bandit convex optimisation is framed as a game between a learner and an adversary where the learner plays actions in a set $K \subset \R^d$, which is 
assumed to be convex, compact and have a nonempty interior.
At the beginning of the game the adversary secretly chooses a 
sequence of convex functions $\ell_1,\ldots,\ell_n : K \to [0,1]$.
The learner and adversary then interact over $n$ rounds. In round $t$ the learner chooses an action $A_t \in K$, suffers loss $\ell_t(A_t)$ and observes
$\ell_t(A_t) + \varepsilon_t$ with the noise $\varepsilon_t$ assumed to be conditionally subgaussian (defined below).
The goal in bandit convex optimisation is to control the regret:
\begin{align*}
\Reg_n = \sup_{x \in K} \sum_{t=1}^n \left( \ell_t(A_t) - \ell_t(x)\right)\,.
\end{align*}
We present a computationally efficient algorithm for bandit convex optimisation. The principle challenge is that almost all of the machinery for optimisation is based
on access to gradients, but in our setting the learner only observes the value of the loss at a single point.
Our algorithm is based on the bandit version of online Newton step for \emph{stochastic unconstrained} bandit convex optimisation by \cite{LG23} with some crucial new ingredients: \textit{(1)} we modify the Minkowski projection proposed by \cite{mhammedi2022efficient}
to show that a bounded convex function on $K$ can be approximately extended to all of $\R^d$ in a such a way
that the function can still be queried via the bandit model 
and \textit{(2)} in the adversarial setting we employ an improved version of the restarting condition used by \cite{SRN21}. 
Our main results are the following regret bounds.

\begin{theorem}\label{thm:main}
There exists an algorithm such that with probability at least $1 - \delta$,
\begin{align*}
    \Reg_n \leq d^{3.5} \sqrt{n} \polylog(n, d, 1/\delta) \,.
\end{align*}
Furthermore, the algorithm is computationally efficient given a membership oracle for $K$.
\end{theorem}

The regret bound can be improved in the stochastic setting, where $\ell_t = \ell$ is the same in every round:

\begin{theorem}\label{thm:stoch}
There exists an algorithm such that in the stochastic setting with probability at least $1 - \delta$,
\begin{align*}
    \Reg_n \leq M d^{2} \sqrt{n} \polylog(n, d, 1/\delta) \,,
\end{align*}
where $M \in [d^{-1/2}, d^{-1/4}]$ is a constant that depends on the geometry of $K$ and the desired computational properties of the algorithm.
\end{theorem}

The computational properties depend on the desired value of $M$ and the properties of $K$. 
A sampling oracle for $K$ when queried returns a sample from the uniform distribution on $K$
and a membership oracle for $K$ when queried with input $x \in \R^d$ returns $\sind_K(x)$.
We have the following:
\begin{itemize}
\item Without any assumptions on $K$ you can take $M = d^{-1/2}$ but the algorithm may not be computationally efficient.
\item Given access to sampling and membership oracles for $K$ you can take $M = d^{-1/4}$ and the algorithm is efficient.
\item Given access to sampling and membership oracles for a symmetric $K$ you can take $M = d^{-1/2}$ and the algorithm is efficient.
\end{itemize}
Besides convexity and boundedness, no regularity assumptions on the losses are needed. 
By efficiency we mean polynomial time, though the (usually heavy) sampling oracle is only needed in the initialisation of the algorithm to place $K$ in
(approximately) isotropic position.

\paragraph{Notation}
Given a function $f: \R^d \to \R$ we write $f'(x)$ for its gradient at $x$ and $f''(x)$ for the Hessian.
The Lipschitz constant for $f$ is $\lip(f) = \sup_{x \neq y} (f(x) - f(y)) / \norm{x - y}$.
The directional derivative at $x$ in direction $\nu$ is $\partial f_\nu(x)$.
The ball of radius $r$ is $\ball(r) = \{x \in \R^d : \norm{x} \leq r\}$ and the sphere is $\sphere(r) = \{x \in \R^d : \norm{x} = r\}$.
The vector of all zeros is $\zeros$ and the identity matrix is $\id$, which
will always be of dimension $d$ and $d\times d$, respectively.
The standard euclidean norm is $\norm{\cdot}$. Given a square matrix $A$ we use the notation $\norm{x}^2_A = x^\top A x$. Be warned,
this is only a norm when $A$ is positive definite but occasionally we use the notation for random matrices $A$ that are only
positive definite in expectation.
The operator norm of a matrix $A$ is $\norm{A} = \max_{x \neq \zeros} \norm{Ax} / \norm{x}$.
We use $\bbP$ to refer to the probability measure on some measurable space carrying all the random variables associated with the
learner/environment interaction, including actions, losses, noise and any exogenous randomness introduced by the learner.
The associated expectation operator is $\E$.
We use $L$ to denote a universal logarithmic factor. More concretely,
\begin{align*}
L = C [1 + \log \max(n, d, 1/\delta)]\,,
\end{align*}
where $C$ is a suitably large universal constant.
We also assume that $\delta = O(\polylog(1 / n, 1 / d))$.
The Gaussian with mean $\mu$ and covariance $\Sigma$ is $\cN(\mu, \Sigma)$. All densities are with respect to the Lebesgue measure.

\paragraph{Related work}
For simplicity in this paragraph we ignore logarithmic factors.
The reader is referred to the recent notes by \cite{L24} for an extensive literature review and summary of the current Pareto frontier.
Bandit convex optimisation was initiated by \cite{FK05} and \cite{Kle04}.
For a long time there was speculation about how the minimax regret should depend on the horizon.
\cite{AFHK13} showed that $\poly(d) \sqrt{n}$ regret is possible in the stochastic setting 
where $\ell_t = \ell$ for some fixed (unknown) loss function $\ell$ and the responses are noisy.
Meanwhile in the adversarial setting $\sqrt{n}$ regret is possible using a kind of online Newton step when
the losses are assumed to be both smooth and strongly convex \citep{HL14}. When only boundedness is assumed, then
\cite{BDKP15} showed that $\sqrt{n}$ regret is possible in $d = 1$, which was extended to $\poly(d) \sqrt{n}$ for the general cases by \cite{BE18}.
These approaches are non-constructive. The bounds are established in a non-constructive fashion via Bayesian minimax duality.
The current state of the art is as follows: \textit{(1)} In the adversarial setting the best known bound is $d^{2.5} \sqrt{n}$ by \cite{Lat20-cvx},
who also used minimax duality and hence do not have an algorithm. \textit{(2)} For a polynomial time algorithm for the adversarial setting there
is the kernel-based algorithm by \cite{BEL16} for which the regret is $d^{10.5} \sqrt{n}$.
\textit{(3)} in the stochastic setting a version of the ellipsoid method obtains $d^4 \sqrt{n}$ regret \citep{LG21a,L24}.
\textit{(4)} in the unconstrained stochastic setting under a Lipschitz assumption there is a version of online Newton step for which the regret is $d^{1.5} \sqrt{n}$ \citep{LG23}.

\paragraph{Noise}
Most of the existing work on adversarial bandit convex optimisation does not consider the presence of noise and simply assumes that $\ell_t(A_t)$ is directly observed.
As far as we know all existing algorithms can handle noise with more-or-less no modifications. We include it for two reasons: \textit{(1)} to ensure
that the adversarial problem is a strict generalisation of the stochastic one; \textit{(2)} so that our algorithm can be used for bandit submodular minimisation
via the Lov\'asz extension.
Let $\tilde{Y}_t = \ell_t(A_t) + \epsilon_t$ be the realised loss in round $t$. 
Our assumption is that the noise is conditionally subgaussian: $\E[\epsilon_t |A_1, \tilde{Y}_1, \ldots, A_{t-1}, \tilde{Y}_{t-1}, A_t] = 0$ and $\E[\exp(\epsilon_t^2)|A_1, \tilde{Y}_1, \ldots, A_{t-1}, \tilde{Y}_{t-1}, A_t] \leq 2$.
Note that this is the Orlicz norm definition of subgaussian, which besides constants is equivalent to the definitions based on moment-generating functions 
or raw moments \citep{Ver18}.

\paragraph{Regularity of constraint set}
Remember that a convex body is a compact, convex subset of $\R^d$ with nonempty interior.
The polar body of $K$ is $K^\circ = \{u : \max_{x \in K} \ip{u, x} \leq 1\}$.
Since the only assumptions on the losses are convexity and boundedness, we can assume that $K$ has been suitably positioned, which simply means a change of coordinates
via an affine transformation. There are two questions: \textit{(1)} how $K$ should be positioned and \textit{(2)} can the corresponding affine transformation be found efficiently.
The (half) mean width of the polar body is
\begin{align*}
M(K^\circ) = \int_{\sphere(1)} \pi(x) \d{\rho}(x)\,,
\end{align*}
where $\pi(x) = \inf\{t > 0 : x \in t K\}$ is the Minkowski functional associated with $K$ and
$\rho$ is the rotationally invariant uniform probability measure on $\sphere(1)$. We let $M = \max(d^{-1/2}, M(K^\circ))$.
We need $K$ to be positioned so that $M(K^\circ)$ is as small as possible and
\begin{align}
\ball(1) \subset K \subset 2 \ball( d+1 )\,.
\label{eq:rounded}
\end{align}
There are various computational tradeoffs available. At present this is the current situation:
\begin{enumerate}
\item When $K$ is in isotropic position, then \cref{eq:rounded} holds with tighter constants \citep{kannan1995isoperimetric} and
$M(K^\circ) \leq 1$ holds trivially.
A convex body can be placed in approximate isotropic position with high probability by estimating the covariance of the uniform distribution on $K$, 
which can be done in polynomial time in a variety of ways \citep{lovasz2006simulated}.
This is a technique already employed by \cite{FK05} and we will not dive into details of the necessary approximation accuracy.

\item When $d K^\circ$ is in isotropic position, then \cref{eq:rounded} holds because polarity reverses inclusion.
The positive resolution to the slicing conjecture and Theorem 9.1.1 by \cite{brazitikos2014geometry} shows
that $M(K^\circ) = \tilde O(d^{-1/4})$. 
If $K$ is also symmetric, then \cite{milman2015mean} improved this bound to $M(K^\circ) = \tilde O(d^{-1/2})$.
$K$ can be positioned so that $d K^\circ$ is approximately in isotropic position using uniform samples from the polar body.
Since a membership/separation oracle for $K$ gives a separation/membership oracle for the polar body, standard sampling algorithms can be used
to find an affine transformation so that $d K^\circ$ is in approximately isotropic position.

\item Lastly, when $K$ is positioned so that $d K^\circ$ is in John's position, then \cref{eq:rounded} holds by John's theorem \citep[Remark 2.1.17]{ASG15}.
Furthermore, \cite{barthe1998extremal} proved that the mean width of $K^\circ$ in this position is upper bounded by the mean width of the standard simplex 
and \cite{finch2011mean} showed this is at most $\tilde O(d^{-1/2})$. Even with a separation/membership oracle there is no known algorithm for efficiently
positioning a convex body into John's position.
\end{enumerate}

Summarising, if we ignore computation complexity we can assume $M = \tilde O(d^{-1/2})$. 
If we need to find the position in a computationally efficient manner, then for symmetric $K$ we can take $M = \tilde O(d^{-1/2})$ and for general $K$
the best known bound is $M = \tilde O(d^{-1/4})$.

\paragraph{Distribution theory}
Let $h : \R^d \to \R$ be convex with $\lip(h) < \infty$ and $W$ have law $\cN(\zeros, \id)$ and $X$ have law $\cN(\mu, \Sigma)$ for $\mu \in \R^d$ and positive definite covariance and let $p$ be the density 
with respect to the Lebesgue measure of $\cN(\mu, \Sigma)$.
It is convenient to have an interpretation of $\E[h''(X)]$ even when $h$ is not twice differentiable.
There are at least two natural ways to define this quantity. 
The distributional view is to define
$\E[h''(X)] = \int_{\R^d} h(x) p''(x) \d{x}$, which is an equality using integration by parts when $h$ is twice differentiable.
Alternatively -- and equivalently -- a smoothing approach can be employed by defining $\E[h''(X)] = \lim_{\varrho \to 0} \E[h''_\varrho(X)]$ where $h_\varrho(x) = \E[h(x+\varrho W)]$.
Convexity ensures that $\E[h''(X)] \succeq \zeros$.
For convex $h, g$ we write $h'' \succeq g''$ if $\E[h''(X)] \succeq \E[g''(X)]$ for all non-degenerate Gaussian random elements $X$.

\section{Overview of the analysis}
Here we sketch the analysis and the main ideas involved in proving Theorems~\ref{thm:main} and \ref{thm:stoch}. 
Simply put, the main idea of this work is to design a surrogate loss that allows us to use an algorithm that is designed for $\R^d$ on any domain $K$. 
The unconstrained algorithm we use is a modified version of the algorithm by \citet{LG23}. Like \citet{LG23}, we sample an $X_t$ 
from a Gaussian distribution with mean $\mu_t \in K$ and covariance matrix $\Sigma_t$ and update $\mu_t$ and $\Sigma_t$ using the 
online Newton step algorithm \citep{HAK07} or, equivalently, the exponential weights algorithm with a Gaussian prior
\citep{HEK18}.
As is standard, we assume that $K$ is reasonably well rounded (either in John's position or isotropic position) and our algorithm plays
on $K_\epsilon = (1 - \epsilon) K$ where $\epsilon = \Theta(1/\sqrt{n})$.
The first step in our analysis is to introduce a bandit version of the extension proposed by \citet{mhammedi2022efficient}, which is a convex function $f_t : \R^d \to \R$
such that $f_t = \ell_t$ on $K_\epsilon$. Importantly, $f_t(x)$ can be evaluated by computing $\ell_t(\Pi(x))$ where 
$\Pi(x) = x/\max(1, \pi_\epsilon(x))$ with $\pi_\epsilon$ the Minkowski functional associated with $K_\epsilon$.
Let $x_\star = \argmin_{x \in K_\epsilon} \sum_{t=1}^n \ell_t(x)$. We show that for any $x \in \R^d$, 
\begin{align*}
     \ell_t(\Pi(x)) - \ell_t(x_\star) \leq f_t(x) - f_t(x_\star) \,.
\end{align*}
At least in the stochastic setting one might try to
apply the analysis of \citet{LG23} to the extended loss functions, who study unconstrained stochastic convex bandits. However, the reader may have already noticed 
that we have introduced a new challenge. \cite{LG23} assumed Lipschitz losses and proved a bound that depends polynomially on the norm of the minimiser and Lipschitz constant.
But the Lipschitz constant of the extension $f_t$ is $\Theta(\sqrt{n})$, which renders existing analysis vacuous.
At the same time, \cite{LG23} only studied the stochastic setting, so some new idea is needed to handle the adversarial setting.

\paragraph{Online Newton step}
Remember that online Newton step applied to a sequence of quadratic loss functions $(\hat q_t)_{t=1}^n$ plays
\begin{align*}
\mu_t = \argmin_{x \in K_\epsilon} \left[\frac{1}{2 \sigma^2} \norm{x}^2 + \eta \sum_{s=1}^{t-1} \hat q_s(x)\right] \,.
\end{align*}
Given $x \in K_\epsilon$ let
\begin{align*}
\hqReg_n(x) = \sum_{t=1}^n \left(\hat q_t(\mu_t) - \hat q_t(x)\right)
\end{align*}
be the regret associated with the quadratic losses $(\hat q_t)_{t=1}^n$.
This is bounded by
\begin{align}
\norm{x_{\tau+1} - x_\star}^2_{\Sigma_{\tau+1}} \leq -\eta \hqReg_\tau(x) + \frac{\diam(K)^2}{2  \sigma^2} + \frac{\eta^2}{2} \sum_{t=1}^\tau \norm{g_t}^2_{\Sigma_{t+1}}\,,
\label{eq:ons}
\end{align}
where $g_t = \hat q_t'(\mu_t)$ and $H_t = \hat q_t''(\mu_t)$ and $\Sigma_t^{-1} = \frac{1}{\sigma^2}\id + \eta \sum_{s=1}^{t-1} H_t$.
When applied to convex bandits it is usual to let $\hat q_t$ be an unbiased estimator of a quadratic approximation of $f_t$ and this is our approach as well.
In order to estimate $g_t$ and $H_t$ the algorithm needs to sample its meta action $X_t$ from a Gaussian $\cN(\mu_t, \Sigma_t)$.

\paragraph{Properties of quadratic surrogate}
Let $q_t = \E_{t-1}[\hat q_t]$ where $\E_{t-1}$ conditions on the history until the start of round $t$.
This function is close to $f_t$ in the sense that provided $\norm{x_\star - \mu_t}_{\Sigma_t^{-1}} \lesssim \frac{1}{\lambda}$ for user-defined constant $\lambda > 0$, then
\begin{align}
\E_{t-1}[f_t(X_t)] - f_t(x_\star) \lesssim
q_t(\mu_t) - q_t(x_\star) + \frac{1}{\lambda} \tr(q''_t(\mu_t) \Sigma_t) 
\label{eq:approx}
\end{align}
Log-determinant arguments and the definition of $\Sigma_t$ means that
\begin{align*}
\frac{1}{\lambda} \sum_{t=1}^n \tr(q''_t(\mu_t) \Sigma_t) = \tilde O\left(\frac{d}{\lambda}\right)\,.
\end{align*}
The parameter $\lambda$ must be balanced carefully. On the one hand, we would like it to be large to control the expression above.
On the other hand, it must be small enough that $\norm{x_\star - \mu_t}_{\Sigma_t^{-1}} \lesssim \frac{1}{\lambda}$ holds for all $t$ with high probability.

\paragraph{Challenge of large losses}
The main problem is that $g_t$ and $H_t$ depend on the magnitude of the extended loss function and when the meta algorithm plays outside of $K_\epsilon$ these
may not be bounded in $[0,1]$. More concretely, once the definitions of $g_t$ have been given, it will be straightforward to show that
\begin{align*}
\norm{g_t}^2_{\Sigma_{t+1}} = \tilde O(d Y_t^2)\,,
\end{align*}
where $Y_t$ is the loss of the extension that the learner observes in round $t$.
In normal circumstances for bounded loss functions $Y_t \in [0,1]$ but because the extension grows very rapidly outside of $K_\epsilon$, this need not be true.
Nevertheless, we are still able to show that $\sum_{t=1}^n Y_t^2 = \tilde O(n)$ with high probability.
The key idea is to use the fact that $f_t$ is minimised on $K_\epsilon$ and grows rapidly outside $K_\epsilon$.
Hence any sensible algorithm should not play often far outside of $K_\epsilon$.
The technical tools we use to prove this are quite interesting.
Since $f_t$ is bounded on $K_\epsilon$ and grows rapidly outside, it effectively has vary large curvature near the boundary of $K_\epsilon$
This means that if the distribution of $X_t$ has a high probability of playing outside $K_\epsilon$, then the estimated Hessian $H_t$ will be large in expectation
and consequentially $\Sigma_t$ will rapidly decrease. Formally we use a Poincar\'e-like inequality for convex functions (Proposition~\ref{prop:poincare-ish}).
Having established that $\sum_{t=1}^n Y_t^2 = \tilde O(n)$ with high probability, by \cref{eq:ons} and \cref{eq:approx} it follows that whenever 
$\norm{\mu_t - x_\star}_{\Sigma_t^{-1}} \lesssim \lambda^{-1}$ for all $t \leq \tau$, then
\begin{align}
\norm{\mu_{\tau+1} - x_\star}^2_{\Sigma_{t+1}^{-1}} \lesssim -\eta \hqReg_\tau(x_\star) + \frac{\diam(K)^2}{2 \sigma^2} + \tilde O\left(\frac{d}{\lambda} + n d \eta^2\right) = (\star)\,.
\label{eq:star}
\end{align}
For complicated reasons, $\sum_{t=1}^n Y_t^2 = \tilde O(n)$ can only be established if $\sigma M\sqrt{d} = \tilde O(1)$ with $M = \max(d^{-1/2}, M(K^\circ))$.
Notice we are now well placed to initiate an induction. Since the regret with respect to the quadratics cannot be too negative, the above bound can be used to simultaneously
bound the regret and prove by induction that $\norm{\mu_t - x_\star}_{\Sigma_t^{-1}} \lesssim \lambda^{-1}$ for all $t$.
All that is needed is to choose $\eta$, $\sigma$ and $\lambda$ so that 
\begin{align*}
(\star) = \tilde O\left(\frac{1}{\lambda^2}\right) \,.
\end{align*}
This is achieved by noting that under our regularity assumptions $\diam(K)^2 = O(d^2)$ and choosing 
\begin{align*}
\sigma &= \tilde \Theta\left(\frac{1}{M \sqrt{d}}\right) &
\lambda &= \tilde \Theta\left(\frac{1}{M d^{3/2}}\right) &
\eta &= \tilde \Theta\left(\frac{1}{\lambda} \sqrt{\frac{1}{nd}}\right) \,.
\end{align*}

\paragraph{Adversarial setting}
In the adversarial setting the regret can be negative, which means that \cref{eq:star} cannot be used anymore to drive an induction. 
At the same time, when the regret is negative we could simply restart the algorithm.
The challenge is to detect when to restart the algorithm. This idea has been explored before by \citet{HL16, BEL16, SRN21}. 
We develop a refined restarting mechanism, which is computationally more efficient and has a somewhat simplified analysis relative to prior work. 
We add negative bonuses to the potential of the algorithm, which has the same effect as increasing the learning rate used by \cite{BEL16, SRN21}.
The main reason for the degradation of our regret bound in the adversarial setting is that we need uniform concentration bounds for our surrogate loss estimates.

\paragraph{Summary}
At a very high level, our main technical contributions are as follows:
\begin{enumerate}
\item We show how to reduce constrained bandit convex optimisation to unconstrained bandit convex optimisation using a modification of the extension 
proposed by \cite{mhammedi2022efficient} for full information online learning.
\item The resulting unconstrained problem loses many of the nice properties that exist in the constrained setting. Notably, the loss function is not bounded and only barely
Lipschitz. We develop techniques for handling this issue in adversarial bandit problems.
\item We refine the restarting mechanism by \cite{SRN21} to be more computationally efficient and somewhat simplify its analysis.
\end{enumerate}

The remainder of the paper is organised as follows. We continue by first discussing the related work. After that we introduce the convex extension in Section~\ref{sec:convexextensions} and the surrogate loss in Section~\ref{sec:surrogateloss}. In Section~\ref{sec:algorithm} we describe our algorithm and in Section~\ref{sec:adv} we prove the regret bound for the algorithm.

\section{Convex extensions}\label{sec:convexextensions}
Like that of \cite{LG23}, our algorithm is really designed for unconstrained loss functions where the learner can play anywhere in $\R^d$.
In order to run this algorithm in the constrained setting we employ an extension of the constrained losses that can be evaluated in the bandit model.
The construction is a natural generalisation of extension proposed by \cite{mhammedi2022efficient} for the full information setting.
The main difference is that there the learner is assumed to have gradient access and the gradient is used to extend the function, while
in the bandit setting only (noisy) function values are available and these must be used instead. 
There is a minor problem that for loss functions that are not Lipschitz there need not be any (finite) convex extension at all.
Because of this we extend the loss from a subset of $K$ to all of $\R^d$.
Let $\pi(x) = \inf\{t > 0 : x \in tK\}$ be the Minkowski functional of $K$, which is a convex function such that $x \in K$ if and only if $\pi(x) \leq 1$.
Given $\epsilon \in (0,\frac{1}{2})$, let $\pip(x) = \max(1, \pi_\epsilon(x))$ where $\pi_\epsilon(x) = \pi(x)/(1-\epsilon)$ is the Minkowski functional of $K_\epsilon = (1 - \epsilon) K$.
The function $\pip$ is convex because it is the maximum of two convex functions. Moreover, $\pip$ is $2$-Lipschitz:

\begin{lemma}\label{lem:lipschitz}
$\sup_{\nu \in \sphere(1)} \partial_\nu \pip(x) \leq 2$ for all $x \in \R^d$.
\end{lemma}

\begin{proof}
The Minkowski functional $\pi_\epsilon$ is the support function of the polar body $K^\circ_{\epsilon}$ and hence the subgradients of $\pi_\epsilon(x)$ 
are in $K^\circ_\epsilon$, see, for example, Example 3.1.5 in \citep{YN18}.
Polarity reverses inclusion and by \cref{eq:rounded}, $\ball(1 - \epsilon) \subset K_\epsilon$ 
and therefore $K_\epsilon^\circ \subset \ball(2)$, using that $\epsilon \in (0,\frac{1}{2})$. 
Hence $\partial_\nu \pi_\epsilon(x) \leq 2$ for any $\nu \in \sphere(1)$. 
The result follows from the definition of $\pip(x) = \max(1, \pi_\epsilon(x))$.
\end{proof}
Let $\ell : K \to [0,1]$ be a convex function.
We are now in the position to define an extension of $\ell$ on $K_\epsilon$. Given $x \in \R^d$ define
\begin{align*}
\ext(x) = \pip(x) \ell\left(\frac{x}{\pip(x)}\right) + \frac{\pip(x) - 1}{\epsilon} \,.
\end{align*}
\begin{lemma}\label{lem:extend}
The function $\ext$ satisfies the following:
\begin{multicols}{2}
\begin{enumerate}
\item $\ext(x) = \ell(x)$ for all $x \in K_\epsilon$. \label{lem:extend:ext}
\item $\ext$ is convex on $\R^d$. \label{lem:extend:cvx}
\item $\partial_x \ext(x) \geq 0$ for all $x \notin K_\epsilon$. \label{lem:extend:opt}
\item $\frac{\pip(x) - 1}{\epsilon} \leq \ext(x) \leq 1 + \left(1 + \frac{1}{\epsilon}\right) [\pip(x) - 1]$. \label{lem:extend:bound}
\end{enumerate}
\end{multicols}
\end{lemma}
\begin{remark}\label{rem:extend}
By part~\ref{lem:extend:opt}, for $x \in \partial K_\epsilon$ the function $t \mapsto \ext(t x)$ is non-decreasing for $t \geq 1$, which implies that
a minimiser of $\ext$ is always in $K_\epsilon$. Moreover, since differentiation is linear, this property carries over to sums of extensions.
\end{remark}

\begin{proof}[Lemma~\ref{lem:extend}]
Parts~\ref{lem:extend:ext}, \ref{lem:extend:opt}, and \ref{lem:extend:bound} are immediate from the definitions since $\pip(x) = 1$ whenever $x \in K_\epsilon$ and using the fact that the losses
are bounded in $[0,1]$.
For part~\ref{lem:extend:cvx}, let $J = \{(x, \lambda) : x \in \R^d, \lambda \geq \pi(x)\}$.
The function $g(x, \lambda) = \lambda \ell(x/\lambda)$ is jointly convex
on $J$ \citep[page 35]{Roc15}.
Let $z \in \R^d$. We claim that $\lambda \mapsto g(z, \lambda)$ is $\frac{1}{\epsilon}$-Lipschitz on $[\pip(z),\infty)$.
Since gradients are monotone for convex functions it suffices to check the gradient at $\lambda = \pip(z)$ and as $\lambda \to \infty$.
For the latter case, $\lim_{\lambda \to \infty} \frac{\d}{\d{\lambda}} g(z, \lambda) = \lim_{\lambda \to \infty} [\ell(z/\lambda) - \partial_x \ell(x/\lambda) / \lambda] 
= \ell(\zeros) \leq 1$.
For the former, 
\begin{align*}
\frac{\d{g(z, \lambda)}}{\d{\lambda}}\bigg|_{\lambda = \pip(z)}
&\explan{(1)}\geq \frac{g(z, \pi(z)) - g(z, \pip(z))}{\pip(z) - \pi(z)}   
= \frac{\pi(z) \ell\left(\frac{z}{\pi(z)}\right) - \pip(z)  \ell\left(\frac{z}{\pip(z)}\right)}{\pip(z) - \pi(z)} \\
&\explan{(2)}\geq - \frac{\pip(z)}{\pip(z) - \pi(z)} 
\explan{(3)}\geq - \frac{1}{\epsilon}\,,
\end{align*}
where \textit{(1)} follows from convexity of $\lambda \mapsto g(z, \lambda)$; \textit{(2)} from the fact that $\ell(x) \in [0,1]$ for all $x \in K$
and \textit{(3)} since $\pi_\epsilon(z) = \pi(z) / (1-\epsilon)$ and using the definition of $\pip(z) = \max(1, \pi_\epsilon(z))$.
Let $x, y \in \R^d$ and $z = \half{x} + \half{y}$. By convexity, $\half{1} \pip(x) + \half{1} \pip(y) \geq \pip(z)$ and hence
\begin{align}
g(z, \half{1} \pip(x) + \half{1} \pip(y))
\geq g(z, \pip(z)) - \frac{1}{\epsilon} \left[\half{1}\pip(x) + \half{1} \pip(y) - \pip(z)\right] \,.
\label{eq:extend:lip}
\end{align}
Combining \cref{eq:extend:lip} with joint convexity of $g$ shows that
\begin{align*}
\ext(z)
&= g(z, \pip(z)) + \frac{1}{\epsilon} \left[\pip(z) - 1\right] 
\leq g(\half{x} + \half{y}, \half{1} \pip(x) + \half{1} \pip(y)) + \frac{1}{\epsilon}\left[\half{1} \pip(x) + \half{1} \pip(y) - 1\right] \\
&\leq \half{1} g(x, \pip(x)) + \half{1} g(y, \pip(y)) + \frac{1}{\epsilon}\left[\half{1} \pip(x) + \half{1} \pip(y) - 1\right] 
= \half{1} \ext(x) + \half{1} \ext(y)\,.
\end{align*}
Since $\ext$ is continuous, midpoint convexity implies convexity and hence $\ext$ is convex \citep[Proposition 1.3]{simon2011convexity}.
\end{proof}

\paragraph{Extension and the regret} 
Let $f_t$ be the extension of $\ell_t$ defined almost analogously to $\ext$ above by
\begin{align*}
f_t(x) = \pip(x) \ell_t\left(\frac{x}{\pip(x)}\right) + \frac{2(\pip(x) - 1)}{\epsilon}\,.
\end{align*}
Note that $f_t(x)$ has an additional $v(x) = \frac{\pip(x) - 1}{\epsilon}$ term compared to $\ext(x)$, which is added to give a little nudge for the meta algorithm
to play inside $K$.
Since $v$ is convex, all properties of $\ext$ listed in Lemma~\ref{lem:extend} carry over to $f_t$, except for the upper bound on $\ext$ 
in Lemma~\ref{lem:extend}\ref{lem:extend:bound}, which carries over with a factor 2. 
Our learner will propose a meta-action $X_t \in \R^d$ sampled from a Gaussian distribution $\cN(\mu_t, \Sigma_t)$ with $\mu_t \in K_\epsilon$.
The actual action passed to the environment is
\begin{align*}
A_t &= X_t/\pip(X_t) \in K_\epsilon &
Y_t &= \pip(X_t)\left[\ell_t(A_t) + \epsilon_t\right] + 2 v(X_t) \,.
\end{align*}
Let $\sF_t = \sigma(X_1,{Y}_1,\ldots,X_t,{Y}_t)$ be the $\sigma$-algebra generated by the data observed at the end of round $t$.
Define $\bbP_t = \bbP(\cdot | \sF_{t})$ and $\E_t[\cdot] = \E[\cdot|\sF_{t}]$.
The next lemma (proof in Appendix~\ref{app:regret-reduction}) shows that with high probability the true regret is bounded in terms of the surrogate regret.

\begin{lemma}\label{lem:regret-reduction}
With probability at least $1 - \delta$, 
$\Reg_n \leq \sqrt{nL} + n \epsilon + \max_{x \in K_{\epsilon}} \fReg_n(x)$, where $\fReg_n(x) = \sum_{t=1}^n (\E_{t-1}[f_t(X_t)] - f_t(x))$
\end{lemma}

The rest of the article is focussed on bounding $\fReg_n$ with high probability.
The reader may already notice the principle challenge. In order for Lemma~\ref{lem:regret-reduction} to be useful we need $\epsilon = O(1/\sqrt{n})$.
But the Lipschitz constant of $f_t$ is more-or-less $O(1/\epsilon)$, which is then too large to use the analysis by \cite{LG23}.
As we mentioned, the key is to exploit the fact that $f_t$ is minimised on $K_\epsilon$ and also bounded in $[0,1]$ on $K_\epsilon$, which means any sensible algorithm should not
spend too much time playing outside of $K_\epsilon$ where the loss is not that well behaved.

\section{Surrogate loss}\label{sec:surrogateloss}
The algorithm keeps track of an iterate $\mu_t \in K_\epsilon$ and covariance matrix $\Sigma_t$. Let $p_t$ be the density
$\cN(\mu_t, \Sigma_t)$ and define the surrogate loss function $s_t : \R^d \to \R$ by
\begin{align*}
s_t(z) = 
\int_{\R^d} \left[\left(1 - \frac{1}{\lambda}\right) f_t(x) + \frac{1}{\lambda} f_t((1 - \lambda) x + \lambda z)\right] p_t(\d{x})\,.
\end{align*}
The surrogate $s_t$ and its quadratic approximation $q_t$ defined below have been used extensively in bandit convex optimisation
\citep{BEL16,LG21a,LG23}. Their properties are summarised in detail in a recent monograph \citep{L24}. For completeness we include the essential properties without
proof in Appendix~\ref{app:surrogate}.
The surrogate loss function is not directly observed, but can be estimated by
\begin{align*}
\hat s_t(z) &= Y_t\left[1 - \frac{1}{\lambda} + \frac{R_t(z)}{\lambda}\right] &
R_t(z) &= \frac{p_t\left(\frac{X_t - \lambda z}{1 - \lambda}\right)}{(1 - \lambda)^d p_t(X_t)} \,.
\end{align*}
Everything is nice enough that limits and integrals exchange so that $\E_{t-1}[\hat s'_t(z)] = s'_t(z)$ and
$\E_{t-1}[\hat s''_t(z)] = s''_t(z)$.
Explicit expressions for $\hat s'_t(z)$ and $\hat s''_t(z)$ are given in Appendix~\ref{app:surrogate}.
We let $g_t = \hat s'_t(\mu_t)$ and $H_t =  \hat s''_t(\mu_t)$ and 
\begin{align*}
\hat q_t(x) &= \ip{g_t, x - \mu_t} + \frac{1}{4}\norm{x - \mu_t}^2_{H_t} &
q_t(x) &= \ip{s'_t(\mu_t), x - \mu_t} + \frac{1}{4} \norm{x - \mu_t}^2_{s''_t(\mu_t)} \,.
\end{align*}

\section{Algorithm}\label{sec:algorithm}
The algorithm can be viewed as online Newton step or follow the regularised leader with quadratic loss estimators and a few bolt-on gadgets.
These gadgets are needed because the loss estimators are only well-behaved on a small region, which in the adversarial setting necessitates 
a restarting procedure. There are three interconnected parts that handle the behaviour of the loss estimators and restarting:
\begin{enumerate}
\item The algorithm computes its iterates in a decreasing sequence of sets $K_\epsilon = K_0 \supset K_1 \supset \cdots \supset K_n$. These are the sets on which
the surrogate loss function and its quadratic approximation are well-behaved.
\item In the adversarial setting, as the focus region shifts, the algorithm subtracts quadratics from the accumulation of losses. 
This procedure is inspired by \cite{zimmert22b}
and mimics the increasing learning rates used by \cite{BEL16} and \cite{SRN21}. We think the arguments are a little simpler but the outcome is the same.
Previous algorithms used volume arguments for deciding when to increase the learning rate, while our method is based on relatively standard algebraic considerations and is more practical to implement.
\item Lastly, in the adversarial setting the algorithm restarts if the regret with respect to estimated surrogate losses is sufficiently negative.
\end{enumerate}

\lstset{emph={or,to,for,then,end,for,if,input,else},emphstyle=\color{blue!100!black}\textbf}
\begin{algorithm}
\begin{spacing}{1.1}
\begin{lstlisting}[mathescape=true,escapechar=\&,numbers=left,xleftmargin=5ex]
input $n$, $\eta$, $\lambda$, $\sigma$, $\gamma$ and $K_0 = K_{\epsilon}$
for $t = 1$ to $n$
  let $\Phi_{t-1}(x) = \frac{1}{2\sigma^2} \norm{x}^2 + \sum_{u=1}^{t-1} \flat_u(x) + \eta \sum_{u=1}^{t-1} \hat q_u(x)$ 
  compute $\mu_t = \argmin_{x \in K_{t-1}} \Phi_{t-1}(x)$ and $\Sigma_t^{-1} = \Phi_{t-1}''(\mu_t)$
  sample $X_t \sim \cN(\mu_t, \Sigma_t)$
  play $A_t = \frac{X_t}{\pip(X_t)}$ and observe $Y_t = \pip(X_t)[\ell_t(A_t) + \epsilon_t] + 2 v(X_t)$
  $K_t =  K_{t-1} \cap \{x : \norm{x - \mu_t}^2_{\Sigma_t^{-1}} \leq \Fmax\}$
  if in the adversarial setting:
  compute $z_t = \argmin_{z \in \R^d} \sum_{s=1}^{t-1} \sind(\flat_s \neq \zeros) \norm{z - \mu_s}_{\Sigma_s^{-1}}^2 \label{line:z}$
  $\label{line:flat}\displaystyle \flat_t(x) = \begin{cases} 
  0 & \text{if } \sum_{s=1}^{t-1} \sind(\flat_s \neq \zeros) \norm{z_t - \mu_s}^2_{\Sigma_s^{-1}} \geq \frac{\Fmax}{24} \\
  -\gamma \norm{x - \mu_t}^2_{\Sigma_t^{-1}} & \text{if } \norm{\cdot}^2_{\Sigma_t^{-1}} \not \preceq \sum_{s=1}^{t-1} \sind(\flat_s \neq \zeros) \norm{\cdot}_{\Sigma_s^{-1}}^2 \\
  -\gamma \norm{x - \mu_t}^2_{\Sigma_t^{-1}} & \text{if } \norm{\mu_t - z_t}_{\Sigma_t^{-1}}^2 \geq \frac{\Fmax}{3} \\
  0 &\text{otherwise}\,.
  \end{cases}$
  if $\max_{y \in K_{t}} \eta \sum_{u=1}^t \left(\hat s_u(\mu_u) - \hat s_u(y)\right) \leq -\frac{\gamma \Fmax}{32}$ then restart algorithm $\label{line:restart}$
  end if 
end for
\end{lstlisting}
\end{spacing} 
\caption{}
\label{alg:cvx:adv}
\end{algorithm}

\paragraph{Computation}
The algorithm is reasonably practical, but some parts may be non-trivial to implement. We describe the key parts below:
\begin{enumerate}
\item Finding the iterate $\mu_t$ requires solving a constrained convex quadratic optimisation problem, which should be somewhat practical with Newton's method.
Note the constraints are defined by an intersection of ellipsoids, of which there could be up to $O(n)$ many. This can likely be reduced dramatically by some sort of doubling trick. In the stochastic setting one could only project on $K_{\epsilon}$ and the same regret bound would hold. 
\item Sampling from a Gaussian requires access to random noise and a singular value decomposition.
\item Finding $z_t$ is another convex quadratic program but this time unconstrained and hence straightforward.
\item Testing the conditions for adding a bonus requires an eigenvalue decomposition of a symmetric matrix.
\item As written in Algorithm~\ref{alg:cvx:adv}, testing the condition for a restart involves minimising a (possibly) non-convex function over the focus region. 
Fortunately, this optimisation does not need to be done exactly and the margin for error dominates the amount of non-convexity. 
In Appendix~\ref{app:opt_quad} we show that a suitable approximation can be found using convex programming by adding a quadratic to the objective
and optimising the resulting convex program. This procedure does not impact the regret.
\end{enumerate}

\paragraph{Constants}
The algorithm is tuned and analysed using a number of interconnected constants.
The correct tuning in the adversarial setting is
\begin{align*}
\lambda &= \frac{1}{d^3 L^5} &
\gamma &= \frac{1}{4d L}  &
\eta &= \sqrt{\frac{d}{nL^3}}  &
\sigma^2 &= \frac{1}{d^2} &
\epsilon &= \frac{d^{3.5} L^{8.5}}{\sqrt{n}} &
\Fmax &= d^5 L^8 \,.
\end{align*}
In the stochastic setting we choose
\begin{align*}
\gamma =  0 && \eta = \frac{M d}{\sqrt{n}} && \lambda = \frac{5}{Md^{3/2} L^3} && \sigma^2 = \frac{1}{16 M^2 d L^3} && \epsilon = \frac{M d^{2} L^5}{\sqrt{n}} && \Fmax = 25 M^2 d^3 L^5\,. 
\end{align*}

\newcommand{\optst}{x_{\star,t}^{\small s}}
\newcommand{\opthatst}{x_{\star,t}^{\small\hat{s}}}

\section{Proof of Theorems~\ref{thm:main} and \ref{thm:stoch}}\label{sec:adv}
We prove Theorems~\ref{thm:main} and \ref{thm:stoch} simultaneously, as both arguments use the same tools. We begin by introducing some 
notation and outlining our strategy.
By the definition of the algorithm the functions $\flat_t$ are quadratic and mostly vanish.
Let $\gamma_t = \gamma \sind(\flat_t \neq\zeros)$ and $m_t = \sum_{u=1}^t \sind(\flat_u \neq \zeros)$.
Let $w_t = (1 - 2\gamma)^{m_t}$, which in the stochastic setting is always $1$ and
in the adversarial setting we will prove it is always in $[1/2,1]$.
Later (in Section~\ref{sec:lem:cov}) we are going to prove that $\Sigma_{t}^{-1}$ satisfies
\begin{align*}
\Sigma_t^{-1} = w_{t-1}\Big[\frac{\id}{\sigma^2} + \eta \sum_{u=1}^{t-1} \frac{H_u}{w_u}\Big] \textnormal{ ~~and we define} && \bar H_t = \E_{t-1}[H_t] &&
\bar \Sigma_t^{-1} = w_{t-1}\Big[\frac{\id}{\sigma^2} + \eta \sum_{u=1}^{t-1} \frac{\bar H_u}{w_u}\Big].
\end{align*}
By Lemma~\ref{lem:unbiased}, $\bar H_t = s_t''(\mu_t)$ is the Hessian of a convex function and therefore positive definite.
Since $\E_{t-1}[H_t] = \bar H_t$ you might expect that $\Sigma_t^{-1}$ and $\bar \Sigma_t^{-1}$ are close and indeed this is the case with high probability.
We need notation for the optimal action for the various surrogate losses: 
\begin{align*}
x_{\star,t} &= \argmin_{x \in K_\epsilon} \sum_{u=1}^t f_u(x) & 
\optst &= \argmin_{x \in K_\epsilon} \sum_{u=1}^t s_u(x)  &
\opthatst &= \argmin_{x \in K_t} \sum_{u=1}^t \hat s_u(x) \,.
\end{align*}
Note the particular choices of domain. In the stochastic setting, we let $x_{\star,1} = x_{\star,2} = \cdots = \xopt$ because all $f_t$ are equal. The regret relative to the various surrogate loss estimates are
\begin{align*}
\fReg_\tau(x) &= \sum_{t=1}^\tau (\E_{t-1}[f_t(X_t)] - f_t(x)) &
\sReg_\tau(x) &= \sum_{t=1}^\tau (s_t(\mu_t) - s_t(x))  \\
\qReg_\tau(x) &= \sum_{t=1}^\tau (q_t(\mu_t) - q_t(x)) &
\hqReg_\tau(x) &= \sum_{t=1}^\tau (\hat q_t(\mu_t) - \hat q_t(x)) \,.
\end{align*}

\begin{definition}\label{def:tau}
We define a stopping time $\tau$ to be the first time that one of the following does not hold: 
\begin{enumerate}
\item In the adversarial setting: $x^s_{\star,\tau} \in K_{\tau + 1}$. In the stochastic setting: $\xopt \in K_{\tau + 1}$. 
\item $\frac{1}{2} \bar \Sigma_{\tau+1}^{-1} \preceq \Sigma_{\tau+1}^{-1} \preceq \frac{3}{2} \bar \Sigma_{\tau+1}^{-1}$. \label{def:tau:S}
\item The algorithm has not restarted at the end of round $\tau$. \label{def:tau:restart}
\end{enumerate}
In case none of these hold, then $\tau$ is defined to be $n$.
\end{definition}

Note that $\Sigma_{t+1}$ and the constraints defining $K_{t + 1}$ are $\sF_{t}$-measurable, which means that $\tau$ is a stopping time 
with respect to the filtration $(\sF_t)_{t=1}^n$. Also note that in the stochastic setting definition~\ref{def:tau}\ref{def:tau:restart} holds by definition of the algorithm. 
The plan is as follows:
\textit{(1)} Use concentration of measure to control all the random events that determine the trajectory of the algorithm. The rest of the analysis
consists of proving that the regret is small on the high-probability event that the algorithm is well-behaved.
\textit{(2)} Derive some basic consequences of the concentration analysis.
\textit{(3)} Prove that the regret is well-controlled until round $\tau$. In the stochastic setting our proof that the regret is well-controlled simultaneously 
implies that $\xopt \in K_{\tau+1}$, at which point the proof of Theorem~\ref{thm:stoch} concludes.
\textit{(4)} Prove that if $\tau \neq n$, then the algorithm restarts.
\textit{(5)} Prove that if the algorithm restarts, then the regret is negative.
By the end of all this and by Lemma~\ref{lem:regret-reduction} we will have established that with probability at least $1 - 15n^2 \delta$,
\begin{align*} 
\Reg_n \leq n \epsilon + \max_{x \in K_\epsilon} \fReg_n(x) \leq  n \epsilon + \begin{cases}
    \frac{\Fmax}{\eta} & \textnormal{ in the stochastic setting} \\
    \frac{\gamma \Fmax}{\eta} & \textnormal{ in the adversarial setting.}
\end{cases}
\end{align*}

\paragraph{Step 1: Concentration}
Because the surrogate estimates are not globally well-behaved we need to ensure that our algorithm behaves in a regular fashion.
This necessitates a moderately tedious concentration analysis, which we summarise here, deferring essential proofs for later. 
The main point is that we define a collection of events and prove they all hold with high probability. The remainder of the argument is carried forward
without probabilistic tools by restricting to sample paths on the intersection of the events defined here.
By a union bound and Lemma~\ref{lem:gaussian-norm}, with probability at least $1 - n\delta$,
\begin{align*}
\tag{\eventGaussian} \label{event:gaussian}
\forall ~t\leq \tau, \quad \norm{X_t - \mu_t}^2_{\Sigma_t^{-1}} \leq d L\,.
\end{align*}
Consider also the event \eventCov{} that
\begin{align*}
\tag{\eventCov}
\frac{1}{2} \bar \Sigma_{\tau+1}^{-1} \preceq \Sigma_{\tau+1}^{-1} \preceq \frac{3}{2} \bar \Sigma_{\tau+1}^{-1}\,.
\end{align*}
Note that by the definition of $\tau$ on \eventCov{} it also holds that $\frac{1}{2} \bar \Sigma_t^{-1} \preceq \Sigma_t^{-1} \preceq \frac{3}{2} \bar \Sigma_t^{-1}$ for
all $t \leq \tau$.
The following lemma shows that \eventCov{} occurs with high probability: 
\begin{lemma}\label{lem:cov}
   $\bbP(\eventCov) \geq 1 - 6n\delta$. 
\end{lemma}
\noindent We also need $\sum_t \hat s_t$ and $\sum_t \hat q_t$ to be suitably well-concentrated on $K_t$. 
Let \eventConc{} be the event that 
\begin{gather}
\tag{\eventConc} \label{event:surrogate-conc}
\begin{aligned}
\max_{x \in K_\tau} \left|\sum_{t=1}^\tau \left(\hat s_t(x) - s_t(x)\right) \right| \leq C_\tau,  
\max_{x \in K_\tau} \left|\sum_{t=1}^\tau \left(\hat q_t(x) - q_t(x)\right) \right| \leq C_\tau,  \left|\sum_{t=1}^\tau \left(\hat q_t(x_\star) - q_t(x_\star)\right) \right| \leq \tilde C_{\tau},
\end{aligned}
\end{gather}
where $C_\tau = \frac{L}{\lambda} \left[\sqrt{d  \sum_{t=1}^\tau \E_{t-1}[Y_t^2]} + 10\sqrt{n}\right]$, and $\tilde C_\tau = \frac{L}{\lambda} \left[\sqrt{\sum_{t=1}^\tau \E_{t-1}[Y_t^2]} + 10\sqrt{n}\right]$.
\begin{lemma}\label{lem:sequential-conc}
$\bbP(\eventConc) \geq 1 - 4n\delta$.
\end{lemma}

\begin{proof}
By a union bound and Lemma~\ref{lem:Ybounds}, with probability at least $1 - n \delta$ for all $t \leq \tau$, $Y_t \leq \frac{10L}{\epsilon}$.
The result follows from Lemmas~\ref{lem:s:conc-s-uniform}, \ref{lem:s:conc-q-uniform}, and \ref{lem:s:conc-q} with a union bound. 
\end{proof}

As we mentioned, a serious challenge in our analysis is that the extended loss functions are bounded in $[0,1]$ in $K_{\epsilon}$ but are unbounded on $\R^d$.
Of course, we hope the algorithm will spend most of its time playing actions $X_t \in K$ but it can happen in the initial rounds that $X_t$ is not in $K$
and then $f_t(X_t)$ could be quite large.
Let \eventFSquared{} be the event that
\begin{align*}
M_\tau = \sum_{t=1}^\tau Y_t^2 \leq 100n\logs^2  \quad \text{and} \quad
\bar M_\tau = \sum_{t=1}^\tau \E_{t-1}[Y_t^2] \leq 100 n\logs^2 \,.
\tag{\eventFSquared}\label{event:f-squared}
\end{align*}

\begin{lemma}\label{lem:f-squared}
$\bbP(\eventFSquared) \geq 1 - 4n\delta$.
\end{lemma}
Combining all the lemmas with a union bound shows that $\eventGood = \eventGaussian \cap \eventCov \cap \eventConc \cap \eventFSquared$ holds with probability at
at least $1 - 14 n \delta$.
For the remainder we bound the regret on $\eventGood$.
\paragraph{Step 2: Basic bounds}
We can now lay the groundwork for the deterministic analysis.
Here we collect some relatively straightforward properties of the random quantities under event \eventGood{}.
On $t \leq \tau$, $\Sigma_t \preceq 2 \bar \Sigma_t \preceq 4 \bar \Sigma_1 = 4 \sigma^2$ we additionally have
\begin{align*}
\norm{X_t - \mu_t} = \norm{X_t - \mu_t}_{\Sigma_t \Sigma_t^{-1}}
\leq \sqrt{d L\norm{\Sigma_t}}
\leq 2\sigma \sqrt{d L} \,.
\end{align*}
Next, for any $t \leq \tau$,
\begin{align*}
\norm{g_t}^2_{\Sigma_{t+1}}
&\leq 2 \norm{g_t}^2_{\bar \Sigma_{t+1}}
\leq 4 \norm{g_t}^2_{\bar \Sigma_t}
\leq 6 \norm{g_t}^2_{\Sigma_t} \\
&= \frac{6 R_t(\mu_t)^2 Y_t^2}{(1 - \lambda)^2} \norm{X_t - \mu_t}^2_{\Sigma_t^{-1}} 
\leq \frac{54 Y_t^2}{(1 - \lambda)^2} \norm{X_t - \mu_t}^2_{\Sigma_t^{-1}} 
\leq  100 Y_t^2 dL\,,
\end{align*}
where in the first line we used the definition of $\tau$ and Lemma~\ref{lem:sigma}.
In the second line we used Lemma~\ref{lem:Rtsmall} to bound $R_t(\mu_t) \leq 3$ and \eventGaussian{}.
Combining the above with event \ref{event:f-squared} yields
\begin{align}
\sum_{t=1}^\tau \norm{g_t}^2_{\Sigma_{t+1}} \leq d L M_\tau \leq \frac{1}{2} ndL^4 \,.
\label{eq:basic:g}
\end{align}
The width of the confidence intervals $C_\tau$ is also well-controlled under \eventGood{}.
Recall that
\begin{align*}
C_\tau &=  \frac{L}{\lambda} \left[\sqrt{d \bar M_\tau} + 10\sqrt{n}\right] 
\leq \frac{L}{\lambda}\left[10L\sqrt{nd} + 10\sqrt{n}\right]
\leq \frac{11L^{2}}{\lambda} \sqrt{nd}
\end{align*}
where the first inequality follows from event \ref{event:f-squared}. Similarly, we have that $\tilde C_\tau \leq \frac{11L^{2}}{\lambda} \sqrt{n}$.
The matrix $\Sigma_t^{-1}$ accumulates curvature estimates. The following lemma follows from 
standard trace/log-determinant arguments and its proof is given in Appendix~\ref{app:lem:logdet}. 
\begin{lemma}\label{lem:logdet}
$\sum_{t=1}^\tau \tr(\Sigma_t \bar H_t) \leq \frac{d L}{\eta}$. 
\end{lemma}
The following lemma affords control over the bonuses we add in the adversarial setting. 

\begin{lemma}\label{lem:bonus}
The number of rounds where the bonus is non-zero is at most $m_\tau \leq d L$.
Furthermore, 
suppose that $x \in \R^d$ and $\norm{x - \mu_t}^2_{\Sigma_t^{-1}} \geq \Fmax$ for some $t \leq \tau$. Then, $\sum_{s=1}^t \flat_s(x) \leq -\frac{\gamma \Fmax}{24}$.
\end{lemma}
Since $\gamma = \frac{1}{4dL}$, the lemma also shows that $w_t \in [1/2,1]$.
\paragraph{Step 3: Regret}
We are finally in a position to bound the regret until round $\tau$. We start by bounding the regret in the adversarial setting. 
\paragraph{Regret in the Adversarial Setting} Suppose that $x^s_{\star,\tau} \in K_{\tau}$ and that the algorithm did not restart at the end of round $\tau$. 
By Theorem~\ref{thm:ftrl2},
\begin{align*}
\frac{1}{2} \norm{x^s_{\star,\tau} - \mu_{\tau+1}}_{\Sigma_{\tau+1}^{-1}}^2 
&\leq \frac{\diam(K)^2}{2\sigma^2} + 2\eta^2\sum_{t=1}^\tau \norm{g_t}^2_{\Sigma_{t+1}} + \sum_{t=1}^\tau \flat_t(x^s_{\star,\tau}) - \eta \hqReg_\tau(x^s_{\star,\tau}) \\
\tag{Event \ref{event:surrogate-conc}}
&\leq \frac{\diam(K)^2}{2\sigma^2} + 2\eta^2 \sum_{t=1}^\tau \norm{g_t}^2_{\Sigma_{t+1}} - \eta \qReg_\tau(x^s_{\star,\tau}) + \eta C_\tau \\
\tag{\cref{eq:basic:g}}
&\leq \frac{\diam(K)^2}{2\sigma^2} + \eta^2 ndL^4 - \eta \qReg_\tau(x^s_{\star,\tau}) + \eta C_\tau \\
\tag{\cref{eq:rounded} and Lemma~\ref{lem:q:lower2}}
&\leq \frac{4(d+1)^2}{\sigma^2} + \eta^2 ndL^4 + \eta C_\tau - \eta \sReg_\tau(x^s_{\star,\tau}) \\
\constantCheck
&\leq \frac{\gamma}{4}\Fmax - \eta \sReg_\tau(x^s_{\star,\tau})\,.
\end{align*}
Reordering the above and using that $\eta \fReg_\tau(\xopt) \leq \eta \qReg_\tau(\xopt) + \frac{dL}{\lambda}$ by Lemmas~\ref{lem:s:basic}, \ref{lem:s:lower}, and \ref{lem:logdet} shows that $\max_{x \in K} \fReg_\tau(x) \leq \frac{\gamma \Fmax}{\eta}$.
Furthermore, by definition of $x^s_{\star,\tau}$ we have $ - \sReg_\tau(x^s_{\star,\tau}) \leq - \sReg_\tau(x^{\hat s}_{\star,\tau})$ and consequently
\newcommand{\hsReg}{\Reg^{\hat s}}
\begin{align*}
    \frac{1}{2} \norm{x^s_{\star,\tau} - \mu_{\tau+1}}_{\Sigma_{\tau+1}^{-1}}^2
    &\leq \frac{1}{4} \Fmax -\eta \sReg_\tau(x^{\hat{s}}_{\star,\tau}) \overset{(a)}{\leq} \frac{1}{4}\Fmax - \eta \hsReg_\tau(x^{\hat s}_{\star,\tau}) + \eta C_\tau \overset{(b)}{\leq} \frac{1}{2}\Fmax\,,
\end{align*}
where the $(a)$ is due to event \eventConc{} and $(b)$ is due the assumption that the algorithm did not restart, which means that $- \eta \hsReg_\tau(x^{\hat s}_{\star,\tau}) \leq \frac{\gamma\Fmax}{32}$ by definition. This implies that $x^s_{\star,\tau} \in K_{\tau + 1}$.
On event \eventGood{}, the only way that $\tau \neq n$ is that the algorithm has restarted or
$x^s_{\star,\tau} \notin K_{\tau}$.
Hence we only need to show that if the latter happens, then the algorithm restarts and if the algorithm restarts, then the regret until time $\tau$ is negative.
\paragraph{Regret in the Stochastic Setting} Observe that the regret on $(f_t)$ can be bounded by:
\begin{align}\label{eq:regineqsto}
    0 \leq \eta \fReg_\tau(\xopt) \leq \eta \qReg_\tau(\xopt) + dL\lambda^{-1}\,,
\end{align}
In the stochastic setting, the most important difference is that $f_t$ is constant over time and therefore $x_{\star,\tau} = \xopt$ is also constant. This means that we do not have to restart to make sure that $\xopt$ stays in $K_{\tau}$, and in turn this means that we do not have to resort to uniform concentration bounds.
After the applying the same steps as in the adversarial setting and replacing $x^s_{\star,\tau}$ by $\xopt$ and $C_{\tau}$ by $\tilde{C}_{\tau}$, leads to 
\begin{align*}
    \frac{1}{2} \norm{\xopt - \mu_{\tau+1}}_{\Sigma_{\tau+1}^{-1}}^2 & \leq \frac{4(d+1)^2}{\sigma^2} + \eta^2 ndL^4 + \frac{dL}{\lambda} + \eta \tilde C_\tau - \eta \fReg_\tau(\xopt) \leq \frac{1}{2} \Fmax - \eta \fReg_\tau(\xopt)\,.
\end{align*}
Rearranging and using that $\fReg_\tau(\xopt) \geq 0$ shows that with probability at least $1 - 3n \delta$ we have that $\fReg_\tau(\xopt) \leq \frac{F_{\max}}{2\eta}$ and that $\xopt \in K_{\tau + 1}$.
Thus, a union bound over $T$ completes the proof of Theorem~\ref{thm:stoch}. 
\paragraph{Step 4: Restart condition}
Suppose that $x^s_{\star,\tau} \notin K_{\tau}$. We will show that the algorithm restarts.
Since the surrogate losses $(s_u)$ are convex there exists $y \in \argmin_{x \in K_{\tau}} \sum_{t=1}^\tau s_t(x)$ such that $y \in \partial K_{\tau}$. 
By the definition of $K_\tau$ there exists a $t \leq \tau$ such that $\norm{y - \mu_t}_{\Sigma_t^{-1}}^2 = \Fmax$. Therefore,
\begin{align*}
\eta \sum_{t=1}^\tau \left(\hat s_t(\mu_t) - \hat s_t(x^{\hat s}_{\star,\tau})\right) 
&= \eta \sum_{t=1}^{\tau} \left(\hat s_t(y) - \hat s_t(x^{\hat s}_{\star,\tau})\right) + \eta \sum_{t=1}^\tau \left(\hat s_t(\mu_t) - \hat s_t(y)\right) \\
\tag{Event~\ref{event:surrogate-conc}, def.\ $y$}
&\leq \eta C_\tau + \eta \sum_{t=1}^\tau \left(\hat s_t(\mu_t) - \hat s_t(y)\right) \\
\tag{Event~\ref{event:surrogate-conc}, Lemma~\ref{lem:q:lower2}}
&\leq 3 \eta C_\tau + \eta \sum_{t=1}^\tau \left(\hat q_t(\mu_t) - \hat q_t(y)\right)\\
\tag{Theorem~\ref{thm:ftrl2}}
&\leq 3 \eta C_\tau + \frac{\diam(K)^2}{\sigma^2} + 2\eta^2 \sum_{t=1}^\tau \norm{g_t}^2_{\Sigma_t} + \sum_{t=1}^\tau \flat_t(y) \\ 
\tag{Lemma~\ref{lem:bonus}}
&\leq 3 \eta C_\tau + \frac{\diam(K)^2}{\sigma^2} + 2\eta^2 \sum_{t=1}^\tau \norm{g_t}^2_{\Sigma_t} - \frac{\gamma \Fmax}{16} \\
\tag{\cref{eq:rounded} and \cref{eq:basic:g}}
&\leq 3 \eta C_\tau + \frac{4(d+1)^2}{\sigma^2} + \eta^2dn L^4 - \frac{\gamma \Fmax}{16} \\
\constantCheck
&\leq -\frac{\gamma \Fmax}{32} \,,
\end{align*}
which means that a restart is triggered as required.
\paragraph{Step 5: Regret on restart}
The last job is to prove the regret is negative whenever a restart is triggered. From the proof of Lemma~\ref{lem:Ybounds} we can recover that $f_\tau(x_{\star, \tau-1}) - f_\tau(x_{\star, \tau}) \leq 8L$. Suppose that the restart does happen at the end of round $\tau$. Then, 
\begin{align*}
\eta \fReg_{\tau}(x_{\star, \tau}) & = \eta \fReg_{\tau}(x_{\star, \tau-1}) + \eta \sum_{t=1}^\tau \left(f_t(x_{\star, \tau-1}) - f_t(x_{\star, \tau})\right)\\ \tag{def.\ of $x_{\star,\tau-1}$}
& \leq \eta \fReg_{\tau}(x_{\star, \tau-1}) + \eta \left(f_\tau(x_{\star, \tau-1}) - f_\tau(x_{\star, \tau})\right) \\ 
& \leq \eta \fReg_{\tau}(x_{\star, \tau-1}) + 8\eta L \\ 
\tag{Lemma~\ref{lem:s:lower}}
&\leq \eta \sum_{t=1}^\tau \left(s_t(\mu_t) - s_t(x_{\star,\tau-1})\right) + \frac{d L}{\lambda} + 8 \eta L  \\ 
\tag{Event~\ref{event:surrogate-conc}}
&\leq \eta \sum_{t=1}^\tau \left(\hat s_t(\mu_t) - \hat s_t(x_{\star,\tau-1})\right) + \frac{d L}{\lambda} + 2 \eta C_\tau   + 8\eta L  \\
\tag{$x_{\star,\tau-1} \in K_\tau$ and def.\ of $x_{\star,\tau}^{\hat s}$}
&\leq \eta \sum_{t=1}^\tau \left(\hat s_t(\mu_t) - \hat s_t(x^{\hat s}_{\star,\tau})\right) + \frac{d L}{\lambda} + 2 \eta C_\tau  + 8 \eta L  \\
\tag{Restart condition}
&\leq  \frac{d L}{\lambda} + \eta C_\tau  + 8 \eta L - \frac{\gamma\Fmax}{32} \\
\constantCheck
& < 0\,.
\end{align*}
In other words, the regret is negative if a restart occurs.
\paragraph{Proof summary adversarial setting}
In step 1 we showed that with high probability event $\eventGood$ holds and the algorithm behaves in a predictable fashion.
In step 3 we showed that the regret until time $\tau$ is at most $\frac{\gamma \Fmax}{\eta}$.
In step 4 we showed that the only way $\tau \neq n$ under $\eventGood$ is if the algorithm restarts after round $\tau$ and
in step 5 we showed that if a restart occurs the regret is negative.
The theorem follows by summing the regret over the each restart.

\section{Discussion}

\paragraph{Submodular minimisation}
Our algorithm can be applied to bandit submodular minimisation via the Lov\'asz extension, as explained by \cite{hazan2012online},
who prove a bound of $O(n^{2/3})$. All the conditions of our algorithm are satisfied using this reduction and since $K = [0,1]^d$ can be efficiently positioned
to a scaling of John's position our algorithms manages $d^{3.5} \sqrt{n} \polylog(n, d, 1/\delta)$ in the adversarial setting and $d^{1.5} \sqrt{n} \polylog(n, d, 1/\delta)$ in
the stochastic setting. 
\paragraph{Constants}
We have chosen not to give explicit constants for the tuning parameters. The situation is a bit unfortunate. 
The theoretically justified constants will yield a rather small learning rate, a large regret bound and an algorithm that converges rather slowly.
On the other hand, if you choose the learning rate $\eta$ or $\lambda$ too large, then the algorithm may fail catastrophically.

\paragraph{Self-concordant barriers} 
Another natural approach would be to use self-concordant barriers as regularizers. Indeed, this yields sublinear regret bounds. We modified the algorithm of \citet{LG23} to use a $\nu$-self concordant barrier rather than the Euclidean norm as a regularizer and added some other tricks. In the stochastic setting, the best regret bound we could prove was $\nu^{1.5}d^{1.5}\sqrt{n}\polylog(n, d, 1/\delta)$. In the adversarial setting, the best bound we could prove was $\nu^4d^{3.5}\sqrt{n}\polylog(n, d, 1/\delta)$. The major technical challenge came from proving that the hessian of the potential was positive definite: with a self-concordant barrier as the regularizer the hessian at $x$ is no longer independent of $x$. This forced us to prove that the hessian of the potential was proportional for all $x$ in the focus region of the algorithm, which is where the $\nu$-factors come from in the aforementioned regret bounds. Since $\nu = O(d)$ for arbitrary $K$ these bounds are considerably worse than the results presented in this paper.

\paragraph{Acknowledgements}
This work was partially done while DvdH was at the University of Amsterdam supported by Netherlands Organization for Scientific Research (NWO), grant number VI.Vidi.192.095. 
Thank you also to our colleagues Tim van Erven, Wouter Koolen and Andr\'as Gy\"orgy for many useful discussions.

\appendix
\crefalias{section}{appendix}
\bibliographystyle{abbrvnat}
\bibliography{all.bib}

\newpage


\section{Proof of Lemma~\ref{lem:bonus}}

Let $E_t = \{s \leq  t : \flat_s \neq \zeros\}$ and
\begin{align*}
c_t = \sum_{s \in E_{t-1}} \norm{\mu_s - z_t}_{\Sigma_s^{-1}}^2 \,.
\end{align*}
We proceed in two steps. First we bound the number of non-zero bonus terms and then we show that the bonuses are sufficiently negative.

\paragraph{Step 1: Number of non-zero bonuses}
For the second part of the lemma we need to upper bound the number of times that $\flat_t \neq \zeros$.
By standard log-determinant arguments,
\begin{align*}
\sum_{t=1}^\tau \sind\left(\flat_t \neq \zeros \text{ and } \norm{\cdot}^2_{\Sigma_t^{-1}} \not \leq \sum_{s\in E_{t-1}} \norm{\cdot}_{\Sigma_s^{-1}}^2\right) \leq \frac{d \logs}{2}\,.
\end{align*}
In more detail, 
let $U_t^{-1} = \sum_{s \in E_t} \Sigma_s^{-1}$ and suppose that $t$ is a round 
where $\flat_t \neq \zeros$ and $\Sigma_t^{-1} \not \preceq U_{t-1}^{-1}$.
Then
\begin{align*}
U_{t-1}^{-1} \preceq U_t^{-1}
= U_{t-1}^{-1} + \Sigma_t^{-1} \not \preceq 2 U_{t-1}^{-1}  \,,
\end{align*}
where in the first inequality we used the fact that $\Sigma_t^{-1} \succeq \zeros$.
Rearranging shows that
\begin{align}
\id \preceq U_{t-1}^{-1/2} U_t U_{t-1}^{-1/2} \not \preceq 2 \id \,.
\label{eq:U-bound}
\end{align}
Therefore, using the facts that $\log \det A^{1/2} B A^{1/2} = \log \det AB$ and for any positive definite $A$ with $\id \preceq A \not\preceq 2 \id$, $\log \det A \geq \log(2)$,
\begin{align*}
\log(2) \sum_{t\in E_{\tau}} \sind\left(\Sigma_t^{-1} \not \preceq U_{t-1}^{-1}\right)
&\explan{(a)}\leq \log(2) + \sum_{t \in E_{\tau}, t > 1} \log \det\left(U_t^{-1} U_{t-1} \right) \\
&\explan{(b)}= \log(2) + \log \det\left(U_{\tau}^{-1} U_1 \right) \\ 
&\explan{(c)}= \log(2) + \log \det \left(\sigma^2 \sum_{s \in E_{\tau}} \Sigma_s^{-1}\right)  \\
&\explan{(d)}\leq \log(2) + \log \det \left(2\sigma^2 \sum_{s \in E_{\tau}} \bar \Sigma_s^{-1}\right) \\
&\explan{(e)}\leq \log(2) + \log \det \left(\frac{2n\sigma^2}{\delta} \id \right) \\
&\explan{(f)}\leq \frac{d L \log(2)}{2} \,.
\end{align*}
where \texttt{(a)} follows from \cref{eq:U-bound} and the fact that for $\id \preceq A \not \preceq 2 \id$, $\log \det A \geq \log(2)$.
\texttt{(b)} follows by telescoping the sum,
\texttt{(c)} by the definition of $U_t$,
\texttt{(d)} by the definition of the stopping time $\tau$,
\texttt{(e)} since $\bar \Sigma_t^{-1} \preceq \frac{\id}{\delta}$,
\texttt{(f)} from the definition of the constants.
Rearranging shows that 
\begin{align*}
\sum_{t \in E_{\tau}} \sind\left(\norm{\cdot}^2_{\Sigma_t^{-1}} \not \leq \sum_{s \in E_{t-1}} \norm{\cdot}_{\Sigma_s^{-1}}^2\right) \leq \frac{d \logs}{2}\,.
\end{align*}
On the other hand, suppose that $\flat_t \neq \zeros$ and $\norm{\cdot}_{\Sigma_t^{-1}}^2 \leq \sum_{s \in E_{t-1}} \norm{\cdot}_{\Sigma_s^{-1}}^2$. 
By the definition of the algorithm we have $c_t \leq \Fmax/24$ and $\norm{\mu_t - z_t}^2_{\Sigma_t^{-1}} \geq \frac{\Fmax}{3}$ and hence
\begin{align*}
c_{t+1}
&= \min_{y \in \R^d} \left(\norm{y - \mu_t}_{\Sigma_t^{-1}}^2 + \sum_{s \in E_{t-1}} \norm{y - \mu_s}_{\Sigma_s^{-1}}^2\right) \\
&\explan{(a)}\geq \min_{y \in \R^d} \left(\norm{y - \mu_t}_{\Sigma_t^{-1}}^2 + \frac{1}{2} \sum_{s \in E_{t-1}} \norm{y - z_t}_{\Sigma_s^{-1}}^2\right) - \sum_{s \in E_{t-1}} \norm{\mu_s - z_t}_{\Sigma_s^{-1}}^2 \\
&\explan{(b)}\geq \min_{y \in \R^d} \left(\norm{y - \mu_t}_{\Sigma_t^{-1}}^2 + \frac{1}{2} \norm{y - z_t}_{\Sigma_t^{-1}}^2\right) - \sum_{s \in E_{t-1}} \norm{\mu_s - z_t}_{\Sigma_s^{-1}}^2 \\
&\explan{(c)}\geq \frac{1}{4} \norm{z_t - \mu_t}_{\Sigma_t^{-1}}^2 - \sum_{s \in E_{t-1}} \norm{\mu_s - z_t}_{\Sigma_s^{-1}}^2 \\
&\explan{(d)}= \frac{1}{4} \norm{z_t - \mu_t}_{\Sigma_t^{-1}}^2 - c_t \\
&\explan{(e)}\geq \frac{\Fmax}{24}\,.
\end{align*}
where \texttt{(a)} follows from the triangle inequality and the fact that $(a+b)^2 \leq 2a^2 + 2b^2$.
\texttt{(b)} uses the condition that $\norm{\cdot}_{\Sigma_t^{-1}}^2 \leq \sum_{s \in E_{t-1}} \norm{\cdot}_{\Sigma_s^{-1}}^2$.
\texttt{(c)} is another triangle inequality and \texttt{(d)} is the definition of $c_t$ and \texttt{(e)} follows by the assumption that $c_t < \Fmax/8$.
Therefore the total number of rounds when $\flat_t \neq \zeros$ is at most $\frac{1}{2} d \logs + 1 \leq d \logs$.

\paragraph{Step 2: Magnitude of bonuses}
Let $t \leq \tau$ and $x$ satisfy $\norm{x - \mu_t}^2_{\Sigma_t^{-1}} \geq \Fmax$. 
Suppose that $c_t \geq \Fmax/24$.
Then, by the definition of $z_t$,
\begin{align*}
\sum_{s \in E_{t-1}} \norm{x - \mu_s}^2_{\Sigma_s^{-1}} \geq \sum_{s \in E_{t-1}} \norm{\mu_s - z_t}^2_{\Sigma_s^{-1}} \geq \frac{\Fmax}{24}\,.
\end{align*}
And the claim follows. On the other hand, if $c_t < \Fmax/24$ and $\flat_t = \zeros$, then
\begin{align*}
\sum_{s \in E_{t-1}} \norm{x - \mu_s}^2_{\Sigma_s^{-1}}
&\explan{(a)}\geq \frac{1}{2} \sum_{s \in E_{t-1}} \norm{x - z_t}^2_{\Sigma_s^{-1}} - \sum_{s \in E_{t-1}} \norm{z_t - \mu_s}_{\Sigma_s^{-1}}^2 \\
&\explan{(b)}\geq \frac{1}{2} \sum_{s \in E_{t-1}} \norm{x - z_t}^2_{\Sigma_s^{-1}} - \frac{\Fmax}{24} \\
&\explan{(c)}\geq \frac{1}{2} \norm{x - z_t}^2_{\Sigma_t^{-1}} - \frac{\Fmax}{24} \\
&\explan{(d)}\geq \frac{1}{4} \norm{x - \mu_t}^2_{\Sigma_t^{-1}} - \frac{1}{2} \norm{z_t - \mu_t}^2_{\Sigma_t^{-1}} - \frac{\Fmax}{24} \\
&\explan{(e)}\geq \frac{\Fmax}{24}\,,
\end{align*}
where \textit{(a)} follows from the triangle inequality and because $(a + b)^2 \leq 2a^2 + 2b^2$.
\textit{(b)} follows from the assumption that $c_t < \Fmax/24$.
\textit{(c)} follows because $\flat_t = \zeros$ so that $\sum_{s \in E_{t-1}} \norm{\cdot}_{\Sigma_s^{-1}}^2 > \norm{\cdot}_{\Sigma_t^{-1}}^2$.
\textit{(d)} uses the same argument as \textit{(a)}
and \textit{(e)} follows because $\flat_t = \zeros$ so that $\norm{z_t - \mu_t}_{\Sigma_t^{-1}}^2 \leq \Fmax/3$.

\section{Proof of Lemma~\ref{lem:f-squared}}
By the definition of $\tau$, for $t \leq \tau$, 
\begin{align}
\Sigma_t \preceq 2 \bar \Sigma_t \preceq 4 \bar \Sigma_1 = 4 \sigma^2 \id \,.
\label{eq:lem:f-squared-0}
\end{align}
By a union bound and Lemma~\ref{lem:Ybounds}, with probability at least $1 - n \delta$, for all $t \leq \tau$,
\begin{align*}
Y_t \leq 4 L + 2 v(X_t)  \,.
\end{align*}
On this event, using the fact that $(a+b)^2 \leq 2a^2 + 2b^2$,
\begin{align}
M_\tau 
= \sum_{t=1}^\tau Y_t^2 
\leq 32 n L^2 + 8 \sum_{t=1}^\tau v(X_t)^2 \,.
\label{eq:lem:f-squared-1}
\end{align}
By Corollary~\ref{cor:poincare-ish},
\begin{align}
\sum_{t=1}^\tau \E_{t-1}[v(X_t)^2]
&\explan{(a)}\leq n + \sum_{t=1}^\tau \frac{2M \sqrt{d \norm{\Sigma_t}}}{\epsilon} \tr\left(\Sigma_t \E_{t-1}[v''(X_t)]\right) \nonumber \\
&\explan{(b)}\leq n + \sum_{t=1}^\tau \frac{4 \sigma M \sqrt{d}}{\epsilon} \tr\left(\Sigma_t \E_{t-1}[v''(X_t)]\right) \nonumber \\
&\explan{(c)}\leq n + \sum_{t=1}^\tau \frac{4 \sigma M \sqrt{d}}{\epsilon} \tr\left(\Sigma_t \E_{t-1}[f''(X_t)]\right) \nonumber \\
&\explan{(d)}\leq n + \frac{8 \sigma M \sqrt{d}}{\epsilon \lambda } \sum_{t=1}^\tau \tr\left(\Sigma_t s_t''(\mu_t)\right) + \sum_{t=1}^\tau \frac{4 \delta \sigma M \sqrt{d}}{\epsilon} \tr(\Sigma_t + \id) \nonumber \\
&\explan{(e)}\leq \frac{3n}{2} + \frac{8  \sigma M \sqrt{d}}{\epsilon \lambda } \sum_{t=1}^\tau \tr\left(\Sigma_t \bar H_t\right) \nonumber \\
&\explan{(f)}\leq \frac{3n}{2} + \frac{8 \sigma  d^{3/2} ML}{\epsilon \lambda \eta} \nonumber \\
&\explan{(g)}\leq 2 n \,,
\label{eq:lem:f-squared}
\end{align}
where \texttt{(a)} follows from Corollary~\ref{cor:poincare-ish},
\texttt{(b)} from \cref{eq:lem:f-squared-0},
\texttt{(c)} since $f_t'' \succeq v''$,
\texttt{(d)} from Proposition~\ref{prop:hess},
\texttt{(e)} since $s_t''(\mu_t) = \bar H_t$ and naive simplification using $\delta$,
\texttt{(f)} follows from Lemma~\ref{lem:logdet}.
Moving on, since $\mu_t \in K_\epsilon$ and using Lemma~\ref{lem:lipschitz} and the definition $v(x) = (\pip(x) - 1)/\epsilon$ 
we have 
\begin{align*}
v(X_t) \leq v(\mu_t) + \frac{2}{\epsilon} \norm{X_t - \mu_t} = \frac{2}{\epsilon} \norm{X_t - \mu_t}\,.
\end{align*}
Define an event
\begin{align*}
E_t = \left\{ v(X_t)^2 \leq \frac{d\sigma^2 L}{\epsilon^2} \right\}\,.
\end{align*}
By \cref{eq:lem:f-squared-0}, $\Sigma_t \preceq 4 \sigma^2 \id$ for $t \leq \tau$. Hence, by Lemma~\ref{lem:gaussian-norm},
\begin{align*}
\bbP\left(\bigcup_{t=1}^\tau E_t^c\right) \leq \delta \,.
\end{align*}
Therefore, with probability at least $1 - \delta$,
\begin{align*}
\sum_{t=1}^\tau v(X_t)^2 = \sum_{t=1}^\tau \sind_{E_t} v(X_t)^2 \,.
\end{align*}
Hence, by Corollary~\ref{cor:conc}, with probability at least $1 - \delta$ and with $\nu = \frac{\epsilon^2}{d\sigma^2  L}$,
\begin{align*}
\sum_{t=1}^\tau \sind_{E_t} v(X_t)^2 
&\leq 2\sum_{t=1}^\tau \E_{t-1}[v(X_t)^2] + \frac{1}{\nu} \log\left(\frac{1}{\delta}\right) 
\leq 4n + \frac{d\sigma^2 L^2}{\epsilon^2} 
\constantCheck
\leq 5n \,.
\end{align*}
where in the last two steps we used \cref{eq:lem:f-squared} and the definition of the constants.
Combining this with a union bound and \cref{eq:lem:f-squared-1} shows that with probability at least $1 - (2 + n)\delta$,
\begin{align*}
M_\tau 
\leq 32 nL^2 + 8 \sum_{t=1}^\tau v(X_t)^2 
\leq 32 nL^2 + 40n \leq 100 n L^2\,.
\end{align*}
On the same event, by Lemma~\ref{lem:Ybounds} and \cref{eq:lem:f-squared} it follows also that
\begin{align*}
\bar M_\tau 
&= \sum_{t=1}^\tau \E_{t-1}[Y_t]^2
\leq 24 \sum_{t=1}^\tau \E_{t-1}[v(X_t)^2] + 12n
\leq 100 n L^2\,.
\end{align*}

\section{Proof of Lemma~\ref{lem:cov}}\label{sec:lem:cov}
By definition, 
\begin{align*}
\Phi_t(x) 
&= \frac{\norm{x}^2}{2\sigma^2} + \eta \sum_{s=1}^t \hat q_s(x) + \sum_{s=1}^t \flat_s(x)  \\
&= \Phi_{t-1}(x) + \eta \hat q_t(x) + \flat_t(x)  \\
&= \Phi_{t-1}(x) + \eta \hat q_t(x) - \gamma_t \norm{x - x_t}^2_{\Sigma_t^{-1}} \,.
\end{align*}
Note that $\Phi_t$ is quadratic, so we will use $\Phi_t''$ for its Hessian at any point. By definition, $\Sigma_{t+1}^{-1} = \Phi''_t$ and therefore 
\begin{align*}
\Sigma_{t+1}^{-1} = \Phi_t'' = \Phi''_{t-1} + \frac{\eta H_t}{2} - 2\gamma_t \Sigma_t^{-1} = (1 - 2\gamma_t) \Sigma_t^{-1} + \frac{\eta H_t}{2}\,.
\end{align*}
By induction it follows that
\begin{align*}
    \Sigma_t^{-1} = \frac{w_{t-1} \id}{\sigma^2} + w_{t-1} \frac{\eta}{2} \sum_{u=1}^{t-1} \frac{H_u}{w_u} \,.
\end{align*}
Let $E$ be the event $\{Y_t \leq \frac{10L}{\epsilon}\} \cap \eventFSquared{}$. By Lemma~\ref{lem:Ybounds} and Lemma~\ref{lem:f-squared} and a union bound,
$\bbP(E) \geq 1 - 5n\delta$.
Combining this with Lemma~\ref{lem:s:conc-H-sum} and a union bound shows that with probability at least $1 - 6n \delta$ for all $\Sigma^{-1}$ such that
$\Sigma_t^{-1} \preceq \Sigma^{-1}$ for all $t \leq \tau$,
\begin{align}
\eta \left|\sum_{s=1}^\tau \frac{H_s - \bar H_s}{w_s}\right| 
\preceq \eta \lambda L\left(\sqrt{d \bar M_\tau} + d^2 \Ymax\right) \Sigma^{-1} 
\preceq  10 \eta \lambda L^2\left(\sqrt{d n} + \frac{d^2}{\epsilon}\right) \Sigma^{-1}
\preceq \frac{1}{8} \Sigma^{-1}\,,
\label{eq:sigma}
\end{align}
where in the final inequality we used the definition of the constants.
We assume for the remainder that the above holds and prove that on this event the required bounds on the covariance matrices hold, which establishes the lemma.
By the definition of $\tau$, for any $s \leq \tau$,
\begin{align*}
\Sigma_s^{-1}
\preceq 2 \bar \Sigma_{s}^{-1}
\preceq 4 \bar \Sigma_{\tau+1}^{-1}\,,
\end{align*}
where in the first inequality we used the definition of $\tau$
and in the second we used that $\bar H_t \succeq \zeros$ and $w_t \in [1/2,1]$ for all $t \leq \tau$. 
Therefore, by \cref{eq:sigma} with $\Sigma^{-1} = 8 \bar \Sigma_{\tau+1}^{-1}$,
\begin{align*}
\Sigma_{\tau+1}^{-1} 
= \frac{w_\tau \id}{\sigma^2} + \frac{\eta}{2} w_\tau \sum_{s=1}^\tau \frac{H_s}{w_s} 
\preceq \frac{w_\tau \id}{\sigma^2} + \frac{\eta}{2} w_\tau \sum_{s=1}^\tau \frac{\bar H_s}{w_s} + \frac{1}{2}\bar \Sigma_{\tau+1}^{-1} 
= \frac{3}{2} \bar \Sigma_{\tau+1}^{-1}\,.
\end{align*}
Similarly,
\begin{align*}
\Sigma_{\tau+1}^{-1} 
= \frac{w_\tau \id}{\sigma^2} + \frac{\eta}{2} w_\tau \sum_{s=1}^\tau \frac{H_s}{w_s}
\succeq\frac{w_\tau \id}{\sigma^2} + \frac{\eta}{2} w_\tau \sum_{s=1}^\tau \frac{\bar H_s}{w_s} - \frac{1}{2} \bar \Sigma_{\tau+1}^{-1}
= \frac{1}{2} \bar \Sigma_{\tau+1}^{-1}  \,.
\end{align*}

\section{Proof of Lemma~\ref{lem:regret-reduction}}\label{app:regret-reduction}
We start with a standard concentration argument.
Since $\ell_t(A_t) \in [0,1]$, by Hoeffding--Azuma's inequality, with probability at least $1 - \delta$,
\begin{align*}
\left|\sum_{t=1}^n \left(\ell_t(A_t) - \E_{t-1}[\ell_t(A_t)]\right)\right| \leq \sqrt{nL}\,.
\end{align*}
On this event, 
\begin{align}
\Reg_n 
&= \max_{x \in K} \sum_{t=1}^n (\ell_t(A_t) - \ell_t(x)) \nonumber \\
&\leq \sqrt{nL} + \max_{x \in K} \sum_{t=1}^n (\E_{t-1}[\ell_t(A_t)] - \ell_t(x)) \nonumber \\
&= \sqrt{nL} + \max_{x \in K} \sum_{t=1}^n (\E_{t-1}[\ell_t(X_t / \pip(X_t))] - \ell_t(x)) \nonumber  \\
&\leq \sqrt{nL} +  \epsilon n + \max_{x \in K_{\epsilon}} \sum_{t=1}^n (\E_{t-1}[\ell_t(X_t / \pip(X_t))] - \ell_t(x))  \nonumber \\
&\leq \sqrt{nL} + \epsilon n + \max_{x \in K_\epsilon} \sum_{t=1}^n (\E_{t-1}[f_t(X_t)] - f_t(x))  \nonumber \\
&= \sqrt{nL} + n\epsilon + \max_{x \in K_\epsilon} \fReg_n(x) \,,
\end{align}
where the second inequality follows because the losses are bounded in $[0,1]$ on $K$ and using \citep[Proposition 3.7]{L24}.
The last inequality follows from Lemma~\ref{lem:extend} and Remark~\ref{rem:extend}.

\section{Proof of Lemma~\ref{lem:logdet}}\label{app:lem:logdet}

This result is more-or-less a consequence of the standard trace/log-determinant inequality. There is a little delicacy needed in arranging a suitable telescoping sum, however.
Let
\begin{align*}
\tilde \Sigma_t^{-1} = \frac{1}{4} \left[\frac{\id}{\sigma^2} + \sum_{u=1}^{t-1} \bar H_u\right] \,.
\end{align*}
Recall that 
\begin{align*}
\bar \Sigma_t^{-1} = w_{t-1} \left[\frac{\id}{\sigma^2} + \sum_{u=1}^{t-1} \frac{\bar H_u}{w_u}\right] \,,
\end{align*}
where the weights $(w_u)$ are in $[1/2,1]$ almost surely. Hence, for any $t \leq \tau$,
\begin{align}
\tilde \Sigma_t^{-1} &\preceq \frac{1}{2} \bar \Sigma_t^{-1} \preceq \Sigma_t^{-1}\,, &
\tilde \Sigma_t^{-1} &\succeq \frac{1}{8} \bar \Sigma_t^{-1} \succeq \frac{1}{16} \Sigma_t^{-1} \,.
\label{eq:logdet-ineq}
\end{align}
with both of the second inequalities following from the definition of the stopping time $\tau$.
Therefore, by Lemma~\ref{lem:spec}, Lemma~\ref{lem:hess}  
\begin{align}
\eta \norm{\tilde \Sigma_t^{1/2} \bar H_t \tilde \Sigma_t^{1/2}} 
\leq 16 \eta \norm{\Sigma_t^{1/2} \bar H_t \Sigma_t^{1/2}}
\leq \frac{16 \eta\lambda \lip(f_t)}{1-\lambda} \sqrt{d \norm{\Sigma_t}} 
\leq \frac{32 \eta\lambda \lip(f_t)}{1-\lambda} \sqrt{d \sigma^2} \leq 1 \,,
\label{eq:spec}
\end{align}
where in the first inequality we used Lemma~\ref{lem:spec},
the second follows from Lemma~\ref{lem:hess},
the third from Lemma~\ref{lem:sigma} and the last from Lemma~\ref{lem:lip} and
the definition of the constants.
We are now ready to bound the quantity of interest:
\begin{align*}
\sum_{t=1}^\tau \tr(\bar H_t \Sigma_t)
&\stackrel{(a)}\leq \frac{1}{\eta} \sum_{t=1}^\tau \tr(\eta \tilde \Sigma_t^{1/2} \bar H_t \tilde \Sigma_t^{1/2}) 
\explan{(b)}\leq \frac{1}{\eta} \sum_{t=1}^\tau \log \det\left(\id + \eta \tilde \Sigma_t^{1/2} \bar H_t \tilde \Sigma_t^{1/2}\right) \\
&= \frac{1}{\eta} \sum_{t=1}^\tau \log \det\left(\tilde \Sigma_t \left(\tilde \Sigma_t^{-1} + \eta \bar H_t\right)\right) 
= \frac{1}{\eta} \sum_{t=1}^\tau \log \det\left(\tilde \Sigma_t \tilde \Sigma_{t+1}\right) \\ 
&= \frac{1}{\eta} \log \det\left(\tilde \Sigma_1 \tilde \Sigma_{\tau+1}^{-1}\right) 
= \frac{1}{\eta} \log \det\left(\id + 4 \sigma^2 \sum_{u=1}^{t-1} \bar H_u\right) \\ 
&\explan{(c)}\leq \frac{d}{\eta} \log \left(1 + \frac{4 \sigma^2}{d} \sum_{u=1}^{t-1} \tr(\bar H_u)\right) 
\explan{(d)}\leq \frac{d L}{\eta}\,,
\end{align*}
where in \textit{(a)} we used \cref{eq:logdet-ineq} and the fact that for positive definite $A, B, M$ with $A \preceq B$, $\tr(AM) \leq \tr(BM) = \tr(B^{1/2} M B^{1/2})$. 
In \textit{(b)} we used \cref{eq:spec} and the fact that for positive definite $A$ with $\norm{A} \leq 1$, $\tr(A) \leq \log \det(\id + A)$, which in turn
follows from the inequality $x \leq \log(1+x)$ for $x \in [0,1]$. 
\textit{(c)} uses the inequality $\log \det A \leq d \log(\tr(A)/d)$, which is a consequence of the
arithmetic/geometric mean inequality.
Finally, \textit{(d)} follows because $\bar H_t = s_t''(\mu_t) \preceq \frac{1}{\delta} \id$ by Lemma~\ref{lem:sigma}.

\section{Technical lemmas}\label{app:tech}

\begin{lemma}\label{lem:spec}
Suppose that $A \preceq B$ and $M$ is positive definite. Then $\norm{A^{1/2} M A^{1/2}} \leq \norm{B^{1/2} M B^{1/2}}$.
\end{lemma}

\begin{proof}
Note that both $A^{1/2} M A^{1/2}$ and $B^{1/2} M B^{1/2}$ are positive definite. Since $A \preceq B$, it follows that
$A^{1/2} \ball(1) \subset B^{1/2} \ball(1)$. Therefore
\begin{align*}
\norm{B^{1/2} M B^{1/2}}
&= \max_{x \in \ball(1)} x^\top B^{1/2} M B^{1/2} x \\
&= \max_{y \in B^{1/2} \ball(1)} y^\top M y \\ 
&\geq \max_{y \in A^{1/2} \ball(1)} y^\top M y \\
&= \max_{x \in \ball(1)} x^\top A^{1/2} M A^{1/2} x \\
&= \norm{A^{1/2} M A^{1/2}} \,.
\end{align*}
\end{proof}

\begin{lemma}[Theorem 5.2.2, \citealt{Ver18}]\label{lem:lip-conc}
Let $h : \R^d \to \R$ and $\delta \in (0,1)$ and $X$ have law $\cN(\mu, \Sigma)$. Then, with probability at least $1 - \delta$,
\begin{align*}
\left|\E[h(X)] - h(X)\right| \leq C \lip(h) \sqrt{\norm{\Sigma} \log(1/\delta)}
\end{align*}
\end{lemma}

\begin{lemma}\label{lem:gaussian-norm}
Let $W$ have law $\cN(\zeros, \id)$. Then for any $\delta \in (0,1)$, $\bbP(\norm{W}_{2} \geq  2\sqrt{2d/3 \log(2/\delta)}) \leq \delta$. 
\end{lemma}
\begin{proof}
    The proof follows by plugging Fact 11a) from \cite{LG23} into Lemma 12a of the same work (i.e. Proposition 2.7.1 of \cite{Ver18}).
\end{proof}

\begin{lemma}[Proposition 2.7.1, \cite{Ver18}]\label{lem:subgaussian}
Let $X$ be a random variable such that $\E[X] = 0$ and $\E[\exp(X^2)] \leq 2$. Then 
\begin{align*}
\bbP(|X| \geq \sqrt{\log(2/\delta)}) \leq \delta \,.
\end{align*}
\end{lemma}

The following is a standard Bernstein-like concentration inequality:

\begin{lemma}\label{lem:conc}
[Exercise 5.15, \citealt{LS20book}]
Let $X_1,\ldots,X_n$ be a sequence of non-negative random variables adapted to a filtration $(\sF_t)_{t=1}^n$ and $\tau$ a stopping time.
Then, for any $\nu$ such that $\nu X_t \leq 1$ almost surely and $\delta \in (0,1)$, with probability at least $1 - \delta$,
\begin{align*}
\sum_{t=1}^\tau X_t \leq \sum_{t=1}^\tau \E[X_t|\sF_{t-1}] + \nu \sum_{t=1}^\tau \E[X_t^2|\sF_{t-1}] + \frac{\log(1/\delta)}{\nu}\,.
\end{align*}
\end{lemma}

\begin{corollary}\label{cor:conc}
Under the same conditions as Lemma~\ref{lem:conc},
\begin{align*}
\sum_{t=1}^\tau X_t \leq 2 \sum_{t=1}^\tau \E[X_t|\sF_{t-1}] + \frac{\log(1/\delta)}{\nu}\,. 
\end{align*}
\end{corollary}

\begin{proof}
Use the fact that $0 \leq \nu X_t \leq 1$ almost surely to bound $\nu \E[X_t^2 | \sF_{t-1}] \leq \E[X_t|\sF_t]$ and apply Lemma~\ref{lem:conc}.
\end{proof}

The next proposition has the flavour of a Poincar\'e inequality but it exploits convexity in a way that the standard inequality does not.

\begin{proposition}\label{prop:poincare-ish}
Suppose that $h : \R^d \to \R$ is convex, non-negative and $\lip(h) \leq 1$ and $X$ has law $\cN(\mu, \Sigma)$ and $h(\mu) = 0$ and $\norm{\Sigma} \leq 1$. Then,
\begin{align*}
\E[h(X)^2] \leq 1 + \left(\E[h(X)] + \sqrt{\norm{\Sigma} L}\right) \tr(\Sigma \E[h''(X)])\,.
\end{align*}
\end{proposition}

\begin{proof}
Assume without loss of generality that $\mu = \zeros$ and let $E = \{h(X) \geq \E[h(X)] + \sqrt{\norm{\Sigma} L} \}$,
Then, using the fact that $h$ is Lipschitz and $h(\zeros) = 0$,
\begin{align}
\E[h(X)^2 \sind_{E^c}]
\leq \E[\norm{X}^2 \sind_{E^c}] 
\leq \sqrt{\E[\norm{X}^4 \bbP(E^c)]} 
\leq 1 \,,
\label{eq:prop:poincare-1}
\end{align}
where in the last inequality we used Lemma~\ref{lem:lip-conc} and a large enough choice of $L$ and the assumption that $\norm{\Sigma} \leq 1$ to bound 
$\E[\norm{X}^4]$.
On the other hand,
\begin{align}
\E[h(X)^2 \sind_E]
&\leq \left(\E[h(X)] + \sqrt{\norm{\Sigma} L} \right) \E[h(X)] \nonumber \\
&\leq \left(\E[h(X)] + \sqrt{\norm{\Sigma} L} \right) \tr(\Sigma \E[h''(X)]) \,,
\label{eq:prop:poincare-2}
\end{align}
where the final equality follows since
\begin{align*}
\E[h(X)]
\tag*{$h$ convex, $h(\zeros) = 0$}
&\leq \E[\ip{h'(X), X}] \\ 
&= \E[\ip{\Sigma h'(X), \Sigma^{-1} X}] \\
\tag*{integrating by parts}
&= \tr\left(\Sigma \E[h''(X)]\right)\,.
\end{align*}
The claim follows by combining \cref{eq:prop:poincare-1} and \cref{eq:prop:poincare-2}.
\end{proof}

\begin{corollary}\label{cor:poincare-ish}
Suppose that $\mu \in K$ and $X$ has law $\cN(\mu, \Sigma)$ with $\Sigma \succeq \zeros$. Then
\begin{enumerate}
\item $\E[v(X)] \leq \frac{M \sqrt{d \norm{\Sigma}}}{2 \epsilon}$. \label{cor:poincare-ish:E}
\item $\E[v(X)^2] \leq 1 + \frac{2 M  \sqrt{d \norm{\Sigma}}}{\epsilon} \tr\left(\Sigma \E[v''(X)]\right)$. \label{cor:poincare-ish:V}
\end{enumerate}
\end{corollary}

\begin{proof}
Let $W$ have law $\cN(\zeros, \norm{\Sigma} \id)$ and $N$ have law $\cN(\zeros, \norm{\Sigma} \id - \Sigma)$, which is a 
legitimate Gaussian since $\norm{\Sigma} \id - \Sigma \succeq \zeros$.
Let $X$, $N$ and $W$ be independent, which means that $X + N$ has the same law as $W$. By convexity,
\begin{align*}
\E[v(X)] 
&\leq \E[v(X + N)] 
= \E[v(\mu + W)] 
= \frac{1}{\epsilon} \E[\pip(\mu+ W) - 1] \\
&= \frac{1}{\epsilon}\left[\max(\pi(\mu + W), 1) - 1\right] 
\leq \frac{1}{\epsilon}\left[\max(\pi(\mu) + \pi(W), 1) - 1\right] 
\leq \frac{1}{\epsilon} \E[\pi(W)]\,,
\end{align*}
where in the second last inequality we used the fact that the Minkowski functional $\pi$ is subadditive and in the last that $\mu \in K$ so that $\pi(\mu) \leq 1$.
Note that $W / \norm{W}$ and $\norm{W}$ are independent and the former us uniformly distributed on $\sphere(1)$ so that
\begin{align*}
\E[\pi(W)] = \E[\pi(W/\norm{W}) \norm{W}] = \E[\pi(W/\norm{W})] \E[\norm{W}]
\leq \frac{M}{2} \E[\norm{W}] \leq \frac{M \sqrt{d \norm{\Sigma}}}{2}\,, 
\end{align*}
where the first inequality follows from the definition of $M$ and the second by Cauchy-Schwarz and because $\E[\norm{W}^2] = d \norm{\Sigma}$.
This completes the proof of \ref{cor:poincare-ish:E}.
For \ref{cor:poincare-ish:V}, by Proposition~\ref{prop:poincare-ish},
\begin{align*}
\E[v(X)^2] 
&\leq 1 + \left(\E[v(X)] + \frac{\sqrt{L}}{\epsilon}\right) \tr(\Sigma \E[v''(X)]) \\
&\leq 1 + \frac{1}{\epsilon} \left(\frac{M \sqrt{d \norm{\Sigma}}}{2} + \sqrt{L}\right) \tr(\Sigma \E[v''(X)])\,.
\end{align*}
The result follows by naive simplification.
\end{proof}

\begin{lemma}\label{lem:sigma}
Suppose that $s \leq t \leq \tau$. The following hold:
\begin{enumerate}
\item $\bar \Sigma_t \preceq 2 \bar \Sigma_s$. \label{lem:sigma:S1}
\item $\Sigma_t \preceq 4 \sigma^2 \id \preceq \id$. \label{lem:sigma:S2}
\item $\bar \Sigma_t^{-1} \preceq \frac{1}{\delta} \id$.  \label{lem:sigma:sinv}
\item $\bar H_t^{-1} \preceq \frac{1}{\delta} \id$. \label{lem:sigma:H}
\end{enumerate}
\end{lemma}

\begin{proof}
By definition,
\begin{align*}
\bar \Sigma_t^{-1} 
= w_{t-1}\left[\frac{\id}{\sigma^2} + \eta \sum_{u=1}^{t-1} \frac{\bar H_u}{w_u}\right]
\succeq w_{t-1}\left[\frac{\id}{\sigma^2} + \eta \sum_{u=1}^{s-1} \frac{\bar H_u}{w_u}\right]
= \frac{w_{t-1}}{w_{s-1}} \bar \Sigma_s^{-1}
\succeq \frac{1}{2} \bar \Sigma_s^{-1}\,.
\end{align*}
This establishes \ref{lem:sigma:S1}. Part~\ref{lem:sigma:S2} follows form \ref{lem:sigma:S1} and the definition of $\tau$, which yields
$\Sigma_t \preceq 2 \bar \Sigma_t$. The claim follows since $\bar \Sigma_1 = \sigma^2 \id$.
For part~\ref{lem:sigma:sinv} we proceed by induction. 
Suppose that for all $u < t$ that $\Sigma_u^{-1} \preceq \frac{1}{\delta} \id$, which is plainly true for $u = 1$. Then
\begin{align*}
\norm{\bar \Sigma_t^{-1}}
&= \norm{w_{t-1}\left[\frac{1}{\sigma^2} \id + \sum_{u=1}^{t-1} \frac{\bar H_u}{w_u}\right]} \\
&\leq \frac{1}{\sigma^2} + 2 \sum_{u=1}^{t-1} \norm{s''_u(\mu_u)} \\
\tag*{Lemma~\ref{lem:hess}}
&\leq \frac{1}{\sigma^2} + \frac{2 \lambda \lip(f_t)}{1 - \lambda} \sum_{u=1}^{t-1} \sqrt{d \norm{\Sigma_u^{-1}}} \\
&\leq \frac{1}{\sigma^2} + \frac{2 n \lambda \lip(f_t) \sqrt{d/\delta}}{1 - \lambda} \\
&\leq \frac{1}{\delta}\,,
\end{align*}
where in the first inequality we used the fact that the weights $(w_t) \in [1/2, 1]$ and the definition of $\bar H_u = s''_u(\mu_u)$ and
the triangle inequality. 
Part~\ref{lem:sigma:H} follows naively from Lemma~\ref{lem:hess} and part~\ref{lem:sigma:sinv}.
\end{proof}

\begin{lemma}\label{lem:lip}
The Lipschitz constant of the extension is bounded by
\begin{align*}
\lip(f_t) \leq \frac{15d}{\epsilon}\,.
\end{align*}
\end{lemma}

\begin{proof}
By definition,
\begin{align*}
f_t(x) = \pip(x) \ell_t\left(\frac{x}{\pip(x)}\right) + \frac{2(\pip(x) - 1)}{\epsilon}\,.
\end{align*}
Note that $\lip(\pip(x)) \leq 2$ and $\ell_t(x) \in [0,1]$ for all $x \in K$ by assumption.
Furthermore, since $\ball(1) \subset K$, $x + \ball(\epsilon) \subset K$ for all $x \in K_\epsilon$ and
by \citep[Proposition 3.6]{L24}, $\lip_{K_\epsilon}(\ell_t) \leq \frac{1}{\epsilon}$.
Since $f_t$ is defined everywhere, it suffices to bound the magnitude of its gradients outside $K$.
Let $x \notin K$ where $\pip(x) = \pi_\epsilon(x)$.
Since $\pi_\epsilon$ is the support function of $K_\epsilon$ its subgradients are in $K_\epsilon^\circ$.
Let $h \in \sphere(1)$ be arbitrary and $\theta \in K_\epsilon^\circ$ be such that $D \pi_\epsilon(x)[h] = \ip{\theta, h}$.
Then,
\begin{align*}
Df_t(x)[h] 
&= \ip{\theta, h}\left[\frac{2}{\epsilon} + \ell_t\left(\frac{x}{\pip(x)}\right)\right] + D \ell_t\left(\frac{x}{\pip(x)}\right)\left[h - \frac{x \ip{\theta, h}}{\pip(x)}\right] \\
&\leq \norm{\theta}\left[\frac{2}{\epsilon} + 1\right] + \frac{1}{\epsilon} \norm{h - \frac{x \ip{\theta, h}}{\pip(x)}} \\
&\leq \norm{\theta}\left[\frac{2}{\epsilon} + 1\right] + \frac{1}{\epsilon} + \frac{\norm{x} \norm{\theta}}{\epsilon \pip(x)} \\
&\leq 2 \left[\frac{2}{\epsilon} + 1\right] + \frac{1}{\epsilon} + \frac{4(d+1)}{\epsilon} \\
&\leq \frac{15d}{\epsilon}\,,
\end{align*}
where we used the assumption that $K \subset \ball(2(d+1))$ and $\epsilon \in (0,1/2)$ so that $\theta \in K_\epsilon^\circ \subset \ball(1/2)$.
\end{proof}

\begin{lemma}\label{lem:Ybounds}
Suppose that $t \leq \tau$.
The following hold with $\bbP_{t-1}$-probability at least $1 - \delta$:
\begin{enumerate}
\item $|Y_t| \leq 4 L + 2 v(X_t)$.
\item $|Y_t| \leq \frac{8L}{\epsilon}$.
\item $\E_{t-1}[Y_t^2] \leq 24 \E_{t-1}[v(X_t)^2] + 12$.
\end{enumerate}
\end{lemma}

\begin{proof}
Since $t \leq \tau$, by Lemma~\ref{lem:sigma}, $\Sigma_t \preceq \id$.
A union bound combined with the fact that $\pip$ is $2$-Lipschitz (Lemma~\ref{lem:lipschitz}) and the assumption that the noise is subgaussian and Lemma~\ref{lem:lip-conc} shows that
with probability at least $1 - \delta$ 
\begin{align}
\pip(X_t) \leq \E_{t-1}[\pip(X_t)] + \sqrt{\norm{\Sigma_t} L} \leq \E_{t-1}[\pip(X_t)] + \sqrt{L} \qquad \text{and} \qquad 
|\epsilon_t| \leq \sqrt{L}\,.
\label{eq:lem:Ybounds}
\end{align}
For the remainder of the proof we assume the above holds.
By Corollary~\ref{cor:poincare-ish}\ref{cor:poincare-ish:E},
\begin{align*}
\E_{t-1}[\pip(X_t)] \leq 1 + \frac{M \sqrt{d \norm{\Sigma_t}}}{2} \leq 1 + \sigma M \sqrt{d} \leq 2\,,
\end{align*}
which when combined with \cref{eq:lem:Ybounds} implies that $\pip(X_t) \leq 2 \sqrt{L}$.
By definition,
\begin{align*}
Y_t 
&= \pip(X_t) \ell_t(X_t / \pip(X_t)) + \pip(X_t) \epsilon_t + 2 v(X_t) \\
&\leq 2 \sqrt{L} + 2 L + 2 v(X_t) \\
&\leq 4 L + 2 v(X_t) \\
&= 4 L + \frac{2 \pip(X_t) - 2}{\epsilon} \\
&\leq 4 L + \frac{4 \sqrt{L} - 2}{\epsilon} \\
&\leq \frac{8L}{\epsilon}\,.
\end{align*}
This establishes both \textit{(a)} and \textit{(b)}. 
For the last part, using the fact that $(a + b + c)^2 \leq 3 a^2 + 3 b^2 + 3 c^2$,
\begin{align*}
\E_{t-1}[Y_t^2]
&= \E_{t-1}\left[\left(\pip(X_t) \ell_t(X_t / \pip(X_t)) + \pip(X_t) \epsilon_t + 2 v(X_t)\right)^2\right] \\
&\leq 12 \E_{t-1}[v(X_t)^2] + 3\E_{t-1}[\pip(X_t)^2] + 3 \E_{t-1}[\pip(X_t)^2 \epsilon_t^2] \\
&\leq 12 \E_{t-1}[v(X_t)^2] + 6 \E_{t-1}[\pip(X_t)^2] \\
&\leq 12 \E_{t-1}[v(X_t)^2] + 12 \E_{t-1}[(\pip(X_t)-1)^2] + 12 \\
&\leq 24 \E_{t-1}[v(X_t)^2] + 12\,,
\end{align*}
where in the second inequality we used the fact that $\E_{t-1}[\epsilon_t^2|X_t] \leq 1$. The remaining steps follow from naive simplification.
\end{proof}

\section{Properties of the surrogate}\label{app:surrogate}
In this section we collect the essential properties of the surrogate loss and its quadratic approximation.
In one form or another, most of these results appeared in the work by \cite{LG21a,LG23} or (less obviously) \cite{BEL16}.
You can find an extensive explanation in the recent monograph by \cite{L24}.

\paragraph{Basic properties}
We start with the elementary properties of the surrogate. To simplify notation, let $f : \R^d \to \R$ be convex, let $X$ have law $\cN(\mu, \Sigma)$ with 
$\Sigma^{-1} \preceq \frac{1}{\delta} \id$ and set
\begin{align*}
s(z) &= \E\left[\left(1 - \frac{1}{\lambda}\right) f(X) + \frac{1}{\lambda} f((1 - \lambda) X + \lambda z)\right] \\ 
q(z) &= \ip{s'(\mu), z - \mu} + \frac{1}{4} \norm{z - \mu}^2_{s''(\mu)} \,. 
\end{align*}

\begin{lemma}[Lemma 11.2, \citealt{L24}]\label{lem:s:basic}
The following hold:
\begin{enumerate}
\item $s$ is optimistic: $s(z) \leq f(z)$ for all $z \in \R^d$; and \label{lem:s:basic:opt}
\item $s$ is convex and infinitely differentiable. \label{lem:s:basic:cvx}
\end{enumerate}
\end{lemma}

\begin{lemma}[Proposition 11.4, \citealt{L24}]\label{lem:hess} 
For any $z \in \R^d$:
\begin{enumerate}
\item $\norm{s''(z)} \leq \frac{\lambda \lip(f)}{1-\lambda} \sqrt{d \norm{\Sigma^{-1}}}$.
\item $\norm{\Sigma^{1/2} s''(z) \Sigma^{1/2}} \leq \frac{\lambda \lip(f)}{1-\lambda} \sqrt{d \norm{\Sigma}}$.
\end{enumerate}
\end{lemma}

\begin{lemma}[Proposition 11.7, \citealt{L24}]\label{lem:s:lower}
If $\lambda \le \frac{1}{dL^2}$, then
\begin{align*}
    \E[f(X)] \leq s(\mu) + \frac{2}{\lambda} \tr(s''(\mu) \Sigma) + \frac{2 \delta d}{\lambda}.
\end{align*}
\end{lemma}
The quadratic surrogate is close to $s$ on a region around $x$ and consequently:

\begin{lemma}[Proposition 11.6, \citealt{L24}]\label{lem:q:lower2}
For all $z \in \R^d$ such that $\lambda \norm{z - \mu}_{\Sigma^{-1}} \leq 1/L$,
\begin{align*}
s(\mu) - s(z) \leq q(\mu) - q(z) + \frac{\delta}{\lambda^2} \,.
\end{align*}
\end{lemma}

Combining Lemmas~\ref{lem:s:basic}, \ref{lem:s:lower} and \ref{lem:q:lower2} with the definition of $\delta$ and $\lambda$ yields the following. 

\begin{lemma}\label{lem:q:lower}
For all $z \in \R^d$ such that $\lambda \norm{z - \mu}_{\Sigma^{-1}} \leq 1/L$,
\begin{align*}
\E[f(X)] - f(z) \leq q(\mu) - q(z) + \frac{2}{\lambda} \tr(s''(\mu) \Sigma) + \frac{1}{n}\,.
\end{align*}
\end{lemma}

The reason why the quadratic surrogate is a reasonable alternative to the surrogate $s$ is because the surrogate is nearly quadratic:

\begin{proposition}[Proposition 11.3, \cite{L24}]\label{prop:s_hess}
If $\lambda\norm{x-y}_{\Sigma^{-1}} \leq L^{-1/2}$, then 
\begin{align*}
s''(x)\preceq 2 s''(y) + \delta \Sigma^{-1} \,.
\end{align*}
\end{proposition}

Based on this we prove a new result:

\begin{proposition}\label{prop:hess}
If $\lambda \le \frac{1}{dL^2}$, then
$\E[f''(X)] \preceq \frac{2 s''(\mu)}{\lambda} + \delta(\Sigma^{-1} + \id) $.
\end{proposition}

\begin{proof}
Let $X\sim \cN(\mu,\Sigma)$, and $\xi^2=\frac{2}{\lambda} - 1$, which is chosen so that if $Z$ has law $\cN(\mu, \xi^2 \Sigma^2)$ and is independent of $X$, then
$(1-\lambda)X+\lambda Z$ has the same law as $X$. 
Thus, by definition 
\begin{align*}
\E[s''(Z)]=\lambda \E[f''((1-\lambda)X+\lambda Z)]=\lambda \E[f''(X)] \,.
\end{align*}
Define an event $E = \{\lambda \norm{Z - \mu}_{\Sigma^{-1}} \leq L^{-1/2}\}$.
Proposition \ref{prop:s_hess} shows that on the event $E$, $s''(Z)\preceq 2s''(\mu) + \delta \Sigma^{-1}$, which means that 
\begin{align*}
\E[s''(Z)]
&=\E\left[\sind_E s''(Z)\right] + \E\left[\sind_{E^c} s''(Z)\right] \\
&\preceq \delta \Sigma^{-1} + 2 s''(\mu) + \E\left[\sind_{E^c} s''(Z)\right] \\
\tag*{by Lemma~\ref{lem:hess}}
&\preceq \delta \Sigma^{-1} + 2 s''(\mu) + \frac{\lambda \lip(f)}{1-\lambda} \sqrt{d \norm{\Sigma^{-1}}} \id \bbP(E^c) \,.
\end{align*}
We now show that $\bbP(E^c)$ is small.
Let $W$ have law $\cN(\zeros, \id)$. By the definition of $Z$,
\begin{align*}
\bbP(E^c)
&= \bbP\left(\norm{Z - \mu}^2_{\Sigma^{-1}} > \frac{1}{L \lambda^2}\right) \\
&= \bbP\left(\norm{W}^2 > \frac{1}{L \xi^2 \lambda^2}\right) \\
&= \bbP\left(\norm{W}^2 > \frac{1}{(2 - \lambda) \lambda L}\right) \\
\tag*{By Lemma~\ref{lem:gaussian-norm}}
&\leq \exp(-L) \,.
\end{align*}
Therefore,
\begin{align*}
\lambda \E[f''(X)] 
&= \E[s''(Z)] \\ 
&\preceq 2 s''(\mu) + \delta \Sigma^{-1} + \frac{\lambda \lip(f)}{1 - \lambda} \sqrt{d \norm{\Sigma^{-1}}} \exp(-L) \id \\
&\preceq 2 s''(\mu) + \delta (\Sigma^{-1} + \id)  \,,
\end{align*}
where in the final inequality we used the definitions of the constants.
\end{proof}

\paragraph{Estimation}
The surrogate $s$ is not directly observed, but can be estimated from data collected by a bandit algorithm.
Let $Y = f(X) + \epsilon$ where $\E[\epsilon] = 0$ and $\E[\exp(\epsilon^2)] \leq 2$ and let
\begin{align}
\hat s(z) &= \left(1 + \frac{r(X, z) - 1}{\lambda}\right) Y &
\hat q(z) &= \ip{\hat s'(x), z - x} + \frac{1}{4} \norm{z - x}^2_{\hat s''(x)} \,.
\label{eq:s:est}
\end{align}
where $r$ is a change-of-measure function defined by
\begin{align*}
r(X, z) &= \frac{p\left(\frac{X - \lambda z}{1 - \lambda}\right)}{(1 - \lambda)^d p(X)} \,. 
\end{align*}
The gradient and Hessian of $\hat s_t$ are given by
\begin{align*}
\hat s'_t(z) &= \frac{Y_t R_t(z)}{1-\lambda} \Sigma_t^{-1}\left[\frac{X_t - \lambda z}{1 - \lambda} - \mu_t\right] \\
\hat s''_t(z) &= \frac{\lambda Y_t R_t(z)}{(1 - \lambda)^2}\left(\Sigma_t^{-1}
\left[\frac{X_t - \lambda z}{1 - \lambda} - \mu_t\right]\left[\frac{X_t - \lambda z}{1 - \lambda} - \mu_t\right]^\top \Sigma_t^{-1} - \Sigma_t^{-1}\right)\,.
\end{align*}

\begin{lemma}[Lemma 11.10, \citealt{L24}]\label{lem:Rtsmall}
    $R_t(\mu_t) \leq 3$ for all $t$.
\end{lemma}

\begin{lemma}[Propositions 11.6 and 11.7, \citealt{L24}]\label{lem:unbiased}
The following hold:
\begin{multicols}{3}
\begin{enumerate}
\item $\E[\hat s(z)] = s(z)$.
\item $\E[\hat s'(z)] = s'(z)$.
\item $\E[\hat s''(z)] = s''(z)$.
\end{enumerate}
\end{multicols}
\end{lemma}

\paragraph{Sequential concentration}
We also need concentration for sums of surrogate estimates and quadratic surrogate estimates.
Let $X_1,Y_1,\ldots,X_n,Y_n$ be a sequence of random elements adapted to a filtration $(\sF_t)_{t=1}^n$ and assume that
$Y_t = f(X_t) + \epsilon_t$ and assume that $\bbP_{t-1}(X_t = \cdot)$ has law $\cN(\mu_t, \Sigma_t)$ and $\epsilon_t$ is conditionally subgaussian:
\begin{align*}
\E[\epsilon_t | \cF_{t-1}, X_t] &= 0 &
\E[\exp(\epsilon_t^2)|\cF_{t-1}, X_t] \leq 2 \,.
\end{align*}
Let $\tau$ be a stopping time with respect to $(\sF_t)_{t=1}^n$ and let 
\begin{align*}
s_t(z) &= \E\left[\left(1 - \frac{1}{\lambda}\right) f_t(X_t) + \frac{1}{\lambda} f_t((1 - \lambda) X_t + \lambda z)\right] \\
q_t(z) &= \ip{s'_t(\mu_t), z - \mu_t} + \frac{1}{4} \norm{z - \mu_t}^2_{s''_t(z)} \,.
\end{align*}
Then let $\hat s_t$ and $\hat q_t$ be defined as in \cref{eq:s:est} but with $X_t$ instead of $X$ and similarly $Y_t$ for $Y$ and so on. 
We let $\Ymax = \max_{1 \leq t \leq \tau} |Y_t|$.
Lastly, given an $r > 0$, let $K_0(r) = K$ and
\begin{align*}
K_t(r) = \{x \in K_{t-1}(r) : \lambda \norm{x - \mu_t}_{\Sigma_t^{-1}} \leq r\}\,.
\end{align*}

\begin{lemma}[Proposition 11.19, \citealt{L24}]\label{lem:s:conc-s}
Let $z \in \R^d$ and assume that $z \in K_\tau(1/\sqrt{2L})$ almost surely.  
Then, with probability at least $1 - \delta$,
\begin{align*}
\left| \sum_{t=1}^\tau (s_t(z) - \hat s_t(z)) \right| \leq 1 + \frac{1}{\lambda} \left[\sqrt{\bar M_\tau L} + L \Ymax\right]\,.
\end{align*}
\end{lemma}

\begin{lemma}[Proposition 11.19, \citealt{L24}]\label{lem:s:conc-s-uniform}
With probability at least $1 - \delta$,
\begin{align*}
\max_{z \in K_\tau(1/\sqrt{2dL})} \left| \sum_{t=1}^\tau (s_t(z) - \hat s_t(z)) \right| \leq 2 + \frac{1}{\lambda} \left [ \sqrt{d \bar M_\tau L} + d L \Ymax\, \right ].
\end{align*}
\end{lemma}
Similar results hold for the quadratic surrogate approximations:
\begin{lemma}[Proposition 11.20, \citealt{L24}]\label{lem:s:conc-q}
Suppose that $z \in K_\tau(1/\sqrt{2L})$ almost surely.
Then, with probability at least $1 - \delta$,
\begin{align*}
\left| \sum_{t=1}^\tau (q_t(z) - \hat q_t(z))\right| \leq \frac{1}{\lambda} \left[\sqrt{\bar M_\tau L} + L \Ymax\right] \,.
\end{align*}
\end{lemma}
And for the uniform bound:
\begin{lemma}[Proposition 11.20, \citealt{L24}]\label{lem:s:conc-q-uniform}
With probability at least $1 - \delta$,
\begin{align*}
\max_{z \in K_\tau(1/\sqrt{2L})} \left|\sum_{t=1}^\tau (q_t(z) - \hat q_t(z))\right| \leq \frac{L^2}{\lambda} \left[\sqrt{d \bar M_\tau} + d \Ymax\right]\,.
\end{align*}
\end{lemma}
Finally, we need control of the Hessian estimates:

\newcommand{\sP}{\mathscr{P}}

\begin{lemma}[Proposition 11.21, \citealt{L24}]\label{lem:s:conc-Hs}
Given $x \in \R^d$, let
\begin{align*}
S_\tau(x) &= \sum_{t=1}^\tau \hat s_t''(x) &
\bar S_\tau(x) &= \sum_{t=1}^\tau s''_t(x)
\end{align*}
let $\sP$ be the (possibly random) set of positive definite matrices $\Sigma$ such that $\Sigma_t^{-1} \preceq \Sigma^{-1}$ for all $t \leq \tau$.
Then, with probability at least $1 - \delta$, for all $\Sigma^{-1} \in \sP$ and $x \in K_\tau(1/\sqrt{2dL})$,
\begin{align*}
\bar S_\tau(x) - \lambda L^2 \left[\sqrt{d \bar M_\tau} + d^2 \Ymax\right] \Sigma^{-1}
\preceq S_\tau(x) 
\preceq \bar S_\tau(x) + \lambda L^2\left[\sqrt{d \bar M_\tau} + d^2 \Ymax\right] \Sigma^{-1} \,. 
\end{align*}
\end{lemma}

\begin{lemma}[Proposition 11.22, \citealt{L24}]\label{lem:s:conc-H-sum}
Let $\sP$ be the same set as in Lemma~\ref{lem:s:conc-Hs} and
\begin{align*}
S_\tau &= \sum_{t=1}^\tau \hat s_t''(\mu_t) &
\bar S_\tau &= \sum_{t=1}^\tau s''_t(\mu_t) \,.
\end{align*}
Then, with probability at least $1 - \delta$, for all $\Sigma^{-1} \in \sP$,
\begin{align*}
\bar S_\tau - \lambda L^2 \left[\sqrt{d \bar M_\tau} + d^2 \Ymax\right] \Sigma^{-1} \preceq S_\tau
\preceq \bar S_\tau + \lambda L^2 \left[\sqrt{d \bar M_\tau} + d^2 \Ymax\right] \Sigma^{-1}\,.
\end{align*}
\end{lemma}

\section{Follow the regularised leader}

The following is standard:

\begin{theorem}\label{thm:ftrl2}
Suppose that $(\hat f_t)_{t=1}^n$ is a sequence of quadratic functions from $K$ to $\R$ and $(K_t)_{t=0}^n \subset K$ is decreasing and 
\begin{align*}
x_t = \argmin_{x \in K_{t-1}} \underbracket{\frac{\norm{x}^2}{2\sigma^2} + \eta \sum_{u=1}^{t-1} \hat f_u(x)}_{\Phi_{t-1}(x)} 
\quad \text{ and } \quad
\norm{\cdot}_{t\star} = \norm{\cdot}_{\Phi_t''^{-1}}\,.
\end{align*}
Suppose that $(\Phi_t)_{t=1}^n$ are convex. Then, for any $x \in K_n$,
\begin{align*}
\eta \sum_{t=1}^n \left(\hat f_t(x_t) - \hat f_t(x)\right) \leq
\frac{\norm{x}^2}{2 \sigma^2} +  2\eta^2 \sum_{t=1}^n \Vert \hat f_t'(x_t) \Vert^2_{t\star}\,,
\end{align*}
\end{theorem}

\begin{proof}
Let $R(x) = \frac{1}{2\sigma^2} \norm{x}^2$.
By the definition of $\Phi_t$,
\begin{align*}
\sum_{t=1}^n \left(\hat f_t(x_t) - \hat f_t(x)\right) 
&= \frac{1}{\eta} \sum_{t=1}^n \left(\Phi_{t}(x_t) - \Phi_{t-1}(x_t)\right) - \frac{\Phi_{n}(x)}{\eta} + \frac{R(x)}{\eta} \\
&= \frac{1}{\eta}\sum_{t=1}^n \left(\Phi_{t}(x_t) - \Phi_{t}(x_{t+1})\right) + \frac{\Phi_{n}(x_{n+1})}{\eta} - \frac{\Phi_{n}(x)}{\eta} + \frac{R(x) - R(x_1)}{\eta} \\
\tag*{$\Phi_n(x_{n+1}) \leq \Phi_n(x)$}
&\leq \frac{1}{\eta} \sum_{t=1}^n \left(\Phi_{t}(x_t) - \Phi_{t}(x_{t+1})\right) + \frac{R(x) - R(x_1)}{\eta} \\
&\leq 2\eta \sum_{t=1}^n \Vert \hat f'_t(x_t)\Vert^2_{t\star} + \frac{\norm{x}^2}{2 \eta \sigma^2} \,. 
\end{align*}
where the final inequality needs to be justified. 
The difference in objectives can be bounded in two ways using that $\Phi_t$ is convex
and only consists of quadratics. Since $\Phi_t$ is quadratic, its second derivative is constant, which we denote by $\Phi''_t$ and let
$\norm{\cdot}_t = \norm{\cdot}_{\Phi_t''}$.
First, 
\begin{align*}
    \Phi_t(x_t) - \Phi_t(x_{t+1}) 
    &= \ip{\Phi'_t(x_{t+1}), x_t - x_{t+1}} 
        + \frac{1}{2}\norm{x_t - x_{t+1}}^2_{t} \\
    &\geq\frac{1}{2}\norm{x_t - x_{t+1}}^2_{t}\,.
        \tag*{First order optimality}
\end{align*}
On the other hand,
\begin{align*}
    \Phi_t(x_t) - \Phi_t(x_{t+1})
    & \leq \ip{\Phi_t'(x_t), x_t -x_{t+1}}
    \tag*{Convexity}\\
    & = \sip{\eta \hat f'_t(x_t) + \Phi_{t-1}'(x_t), x_t -x_{t+1}}\\
    & \leq \sip{\eta \hat f'_t(x_t), x_t -x_{t+1}} 
    \tag*{First order optimality}\\
    \tag*{Cauchy-Schwarz}
    & \leq \eta \Vert \hat f'_t(x_t)\Vert_{t\star} \norm{x_{t+1} -x_t}_{t} \,.
\end{align*}
Combining both gives
\begin{align*}
    (\Phi_t(x_t) - \Phi_t(x_{t+1}))^2
    & \leq \eta^2\Vert \hat f'_t(x_t)\Vert^2_{t\star} \norm{x_t - x_{t+1}}^2_{t}
     \leq 2\eta^2\Vert \hat f'_t(x_t)\Vert^2_{t\star} (\Phi_t(x_t) - \Phi_t(x_{t+1}))
\end{align*}
Dividing both sides by $\Phi_t(x_t) - \Phi_t(x_{t+1})$ gives the result.
\end{proof}

\section{Optimising the surrogate estimate}\label{app:opt_quad}

In the restart condition of the adversarial version of the algorithm we need to
optimise a possibly non-convex function. This problem can be circumvented 
by adding a quadratic term that ensures that the objective becomes convex again. 
Here, we  show that that is possible and that the overhead of this new objective poses no problem for the restart condition. 

For $s \leq t \leq \tau$ we can bound each covariance matrix as in the proof of
Lemma~\ref{lem:cov}
\begin{align*}
    \Sigma^{-1}_s  
    \preceq 2 \bar \Sigma^{-1}_s 
    \preceq 4  \bar \Sigma^{-1}_{t} 
    \preceq 8 \Sigma^{-1}_{t}\,.
\end{align*}
Lemma~\ref{lem:s:conc-Hs} and $t \leq \tau$ then gives us for all $x \in K_t$ that
\begin{align*}
    \sum_{i=1}^{t}\hat s''(x) 
    + \lambda L^2\left[ \sqrt{d \bar M_\tau}  + d^2 \Ymax \right]   8\Sigma^{-1}_t 
    \succeq \zeros
\end{align*}
We can simplify this a bit by noticing
\begin{align*}
    \lambda L^2\left[ \sqrt{d \bar M_\tau}  + d^2 \Ymax \right] 
    \le \lambda L^2\left[ 10L \sqrt{nd} + \frac{d^2 10 L }{\epsilon}  \right]  
    \le 20 \lambda L^{3}\sqrt{nd} 
.\end{align*}
In particular, the function
\begin{align*}
    \sum_{i=1}^{t} \hat s_i(x) + 160\lambda L^{3}\sqrt{nd}\norm{x - \mu_{t}}^2_{\Sigma^{-1}_{t} } 
    =:
    \sum_{i=1}^{t} \hat s_i(x) + Q_t(x)
\end{align*}
is convex on $K_t$. We immediately get
\begin{align}\label{eq:approx_opt_upper_bound}
    \sum_{i=1}^t \hat s_i(\mu_i) 
      - \min_{x \in K_{t}} \left [\sum_{i=1}^t \hat s_i(x) + Q_t(x) \right]
    \leq
    \sum_{i=1}^t \hat s_i(\mu_i) 
        - \min_{x \in K_{t}}\sum_{i=1}^t \hat s_i(x) 
,\end{align}
Because $Q_t(x)$ is positive. We can also bound the original optimum of the non-convex objective as follows. 
\begin{align*}
    \sum_{i=1}^t \hat s_i(\mu_i) - \sum_{i=1}^t \hat s_i(\opthatst)
    & =
    \sum_{i=1}^t \hat s_i(\mu_i) - \sum_{i=1}^t \hat s_i(\opthatst) - Q_t(\opthatst) +  Q_t(\opthatst)\\
    & \leq \sum_{i=1}^t \hat s_i(\mu_i) 
        - \left[\sum_{i=1}^t \hat s_i(\opthatst) + Q_t(\opthatst) \right ] 
        +  160\lambda L^{3} \sqrt{nd} \Fmax\\
    & \leq \sum_{i=1}^t \hat s_i(\mu_i) 
        - \min_{x \in K_{t}} \left [\sum_{i=1}^t \hat s_i(x) + Q_t(x) \right]
        + 160 \lambda L^{3}\sqrt{nd} \Fmax
\end{align*}
Where have used that $\norm{x - \mu_t}^2_{t} \leq \Fmax$ for all $x \in
K_{t}$. With the parameter settings we have
\begin{align*}
    \eta\lambda L^3\sqrt{nd}\Fmax = \frac{\Fmax}{d^2 L^{2.5}} \leq \gamma \Fmax
\end{align*}
The new restart condition becomes
\begin{align*}
    \max_{y \in K_{t}} \eta 
    \sum_{i=1}^t(\hat s_i(\mu_i) - \hat s_i(y) - Q_t(y)) 
    \leq -\left(\frac{160\Fmax}{d^2L^{2.5}} + \frac{\gamma\Fmax}{32}\right)
\end{align*}
A restart should still be triggered when $x^s_{\star,\tau} \not\in K_\tau$. 
This can be achieved by adjusting the original restart condition 
to $\frac{-\gamma \Fmax}{16}$, retuning the parameters and constants such that 
$\frac{dL}{\lambda} + 2\eta C_\tau + 8 \eta L \le \frac{\gamma \Fmax}{16}$ and Lemma~\ref{lem:bonus}
instead gives $\sum_{i=1}^{\tau}\flat_i(x) \le  -\frac{\gamma \Fmax}{8}$. 
By repeating the derivation in Step~$4$  and \eqref{eq:approx_opt_upper_bound} we
get 
\begin{align*}
\max_{x \in K_{t}}\eta \sum_{t=1}^\tau \left(\hat s_t(\mu_t) 
    - \hat s_t(x) - Q_\tau(x)\right) 
&\leq \max_{x \in K_{t}}\eta \sum_{t=1}^\tau \left(\hat s_t(\mu_t) 
    - \hat s_t(x) \right)  \\
&\leq - \frac{\gamma \Fmax}{16} 
\le -\left(\frac{160\Fmax}{d^2L^{2.5}} + \frac{\gamma\Fmax}{32}\right)
\end{align*}
When a restart is triggered, the regret is still negative
\begin{align*}
\eta \fReg_{\tau}(x_{\star, \tau}) 
& \leq \eta \sum_{t=1}^\tau 
    \left(\hat s_t(\mu_t) - \hat s_t(x^{\hat s}_{\star,\tau})\right) 
    + \frac{d L}{\lambda} 
    + 2 \eta C_\tau  
    + 8\eta L  
    \tag*{Repeat of Step 5}
    \\
& \leq \eta \sum_{t=1}^\tau \left(\hat s_t(\mu_t) 
    - \min_{x \in K_{t}} \left [\sum_{t=1}^\tau \hat s_t(x) + Q_\tau(x) \right]\right) 
    + \eta8\lambda \sqrt{nd} \Fmax
    + \frac{d L}{\lambda} + 2 \eta C_\tau \\
&\leq -\left(\frac{160\Fmax}{d^2L^{2.5}} + \frac{\gamma\Fmax}{32}\right) 
    + \eta 160 \lambda \sqrt{nd}\Fmax 
    + \frac{d L}{\lambda} + 2 \eta C_\tau + 8 \eta L
\tag*{Approximate Restart condition}\\
&\leq  \frac{d L}{\lambda} + 2 \eta C_\tau  + 8 \eta L - \frac{\gamma\Fmax}{16} \\
\constantCheck
&\leq 0\,.
\end{align*}

\end{document}